\documentclass[reqno, a4paper]{amsart}

\usepackage{fixltx2e}
\usepackage[T1]{fontenc}
\usepackage[british]{babel}
\usepackage{amsfonts, amssymb, amsthm, amsmath}
\usepackage[all]{xy}
\usepackage{hyperref}
\usepackage{mathbbol}
\usepackage{calrsfs}
\usepackage{mathabx}
\usepackage{wasysym}
\usepackage{graphicx}

\theoremstyle{plain}
\newtheorem{lemma}[subsection]{Lemma}
\newtheorem{theorem}[subsection]{Theorem}
\newtheorem{proposition}[subsection]{Proposition}
\newtheorem{corollary}[subsection]{Corollary}

\theoremstyle{definition}
\newtheorem{example}[subsection]{Example}
\newtheorem{examples}[subsection]{Examples}
\newtheorem{definition}[subsection]{Definition}
\newtheorem{remark}[subsection]{Remark}

\newenvironment{tfae}
{
\begin{enumerate}}
{\end{enumerate}}

\newcommand{\noproof}{\hfill\qed}
\newcommand{\defn}{\textbf}
\newcommand{\DefEq}{\coloneq}

\newcommand{\join}{\vee}

\renewcommand{\Im}{\mathrm{Im}}

\newcommand{\To}{\Rightarrow}
\newcommand{\gauche}{\langle}
\newcommand{\droite}{\rangle}
\newcommand{\links}{\lgroup}
\newcommand{\rechts}{\rgroup}
\newcommand{\cosmash}{\diamond}
\newcommand{\mus}[2]{\left\links\begin{smallmatrix} #1 & #2 \end{smallmatrix}\right\rechts}
\newcommand{\muss}[3]{\left\links\begin{smallmatrix} {#1} & {#2} & {#3}\end{smallmatrix}\right\rechts}
\newcommand{\normal}{\ensuremath{\lhd}}

\DeclareMathOperator{\cod}{cod}
\DeclareMathOperator{\Ker}{Ker}
\DeclareMathOperator{\op}{op}
\DeclareMathOperator{\Aut}{Aut}

\newcommand{\B}{\ensuremath{\mathcal{B}}}
\newcommand{\C}{\ensuremath{\mathcal{C}}}
\newcommand{\D}{\ensuremath{\mathcal{D}}}
\newcommand{\E}{\ensuremath{\mathcal{E}}}

\newcommand{\EE}{\ensuremath{\mathbb{E}}}
\newcommand{\TT}{\ensuremath{\mathbb{T}}}

\newcommand{\Ab}{\ensuremath{\mathsf{Ab}}}
\newcommand{\Act}{\ensuremath{\mathsf{Act}}}
\newcommand{\Arr}{\ensuremath{\mathsf{Arr}}}
\newcommand{\Cat}{\ensuremath{\mathsf{Cat}}}
\newcommand{\CMon}{\ensuremath{\mathsf{CMon}}}
\newcommand{\Eq}{\ensuremath{\mathsf{Eq}}}
\newcommand{\Ext}{\ensuremath{\mathsf{Ext}}}
\newcommand{\CExt}{\ensuremath{\mathsf{CExt}}}
\newcommand{\Gp}{\ensuremath{\mathsf{Gp}}}
\newcommand{\Gpd}{\ensuremath{\mathsf{Gpd}}}
\newcommand{\Haus}{\ensuremath{\mathsf{Haus}}}
\newcommand{\HComp}{\ensuremath{\mathsf{HComp}}}
\newcommand{\Inh}{\ensuremath{\mathsf{GS}}}
\newcommand{\Mon}{\ensuremath{\mathsf{Mon}}}
\newcommand{\MonC}{\ensuremath{\mathsf{MonC}}}
\newcommand{\Pt}{\ensuremath{\mathsf{Pt}}}
\newcommand{\PXMod}{\ensuremath{\mathsf{PXMod}}}
\newcommand{\RG}{\ensuremath{\mathsf{RG}}}

\newcommand{\SPt}{\ensuremath{\mathsf{SPt}}}
\newcommand{\XMod}{\ensuremath{\mathsf{XMod}}}
\newcommand{\Nil}{\ensuremath{\mathsf{Nil}}}
\newcommand{\Top}{\ensuremath{\mathsf{Top}}}
\newcommand{\Set}{\ensuremath{\mathsf{Set}}}
\newcommand{\Hopf}{\ensuremath{\mathsf{HopfAlg}_{K, \mathsf{coc}}}}

\newcommand{\Atwo}{{\rm (A2)}}
\newcommand{\ACC}{{\rm (ACC)}}

\newcommand{\CI}[1]{{\rm (COI\,#1)}}
\newcommand{\UCE}{{\rm (UCE)}}
\newcommand{\FWACC}{{\rm (FWACC)}}
\newcommand{\LACC}{{\rm (LACC)}}
\newcommand{\NH}{{\rm (NH)}}
\newcommand{\SH}{{\rm (SH)}}

\newcommand{\SSH}{{\rm (SSH)}}

\def\pullback{% with thanks to Valerian Even
 \ar@{-}[]+R+<6pt,-1pt>;[]+RD+<6pt,-6pt>%
 \ar@{-}[]+D+<1pt,-6pt>;[]+RD+<6pt,-6pt>}
 
\def\pushout{%
 \ar@{-}[]+L+<-6pt,1pt>;[]+LU+<-6pt,6pt>%
 \ar@{-}[]+U+<-1pt,6pt>;[]+LU+<-6pt,6pt>}

\def\splitpullback{%
 \ar@{-}[]+R+<6pt,-.51ex>;[]+RD+<6pt,-6pt>%
 \ar@{-}[]+D+<.51ex,-6pt>;[]+RD+<6pt,-6pt>}

\def\skewpullback{%
 \ar@{-}[]+LD+<-6pt,-6pt>;[]+LDD+<-6pt,-15.5pt>%
 \ar@{-}[]+D+<-3pt,-6.5pt>;[]+LDD+<-6pt,-15.5pt>}

\newdir{>>}{{}*!/3.5pt/:(1,-.2)@^{>}*!/3.5pt/:(1,+.2)@_{>}*!/7pt/:(1,-.2)@^{>}*!/7pt/:(1,+.2)@_{>}}
\newdir{ >>}{{}*!/8pt/@{|}*!/3.5pt/:(1,-.2)@^{>}*!/3.5pt/:(1,+.2)@_{>}}
\newdir{ |>}{{}*!/-3.5pt/@{|}*!/-8pt/:(1,-.2)@^{>}*!/-8pt/:(1,+.2)@_{>}}
\newdir{ >}{{}*!/-8pt/@{>}}
\newdir{>}{{}*:(1,-.2)@^{>}*:(1,+.2)@_{>}}
\newdir{<}{{}*:(1,+.2)@^{<}*:(1,-.2)@_{<}}

\begin{document}

\title{Algebraically coherent categories}

\author{Alan S.~Cigoli}
\author{James R.~A.~Gray}
\author{Tim Van~der Linden}

\email{alan.cigoli@unimi.it}
\email{jamesgray@sun.ac.za}
\email{tim.vanderlinden@uclouvain.be}

\address{Dipartimento di Matematica, Universit\`a degli Studi di Milano, Via Saldini 50, 20133 Milano, Italy}
\address{Mathematics Division, Department of Mathematical Sciences, Stellenbosch University, Private Bag X1, Matieland 7602, South Africa}
\address{Institut de Recherche en Math\'ematique et Physique, Universit\'e catholique de Louvain, chemin du cyclotron~2 bte~L7.01.02, 1348 Louvain-la-Neuve, Belgium}

\thanks{The first author's research was partially supported by FSE, Regione Lombardia.
The first two authors would like to thank the Institut de Recherche en Math\'ematique et Physique (IRMP) for its kind hospitality during their respective stays in Louvain-la-Neuve.
The third author is a Research Associate of the Fonds de la Recherche Scientifique--FNRS}

\begin{abstract}
We call a finitely complete category \emph{algebraically coherent} when the change-of-base functors of its fibration of points are \emph{coherent}, which means that they preserve finite limits and jointly strongly epimorphic pairs of arrows.
We give examples of categories satisfying this condition; for instance, coherent categories, \emph{categories of interest} in the sense of Orzech, and (compact) Hausdorff algebras over a semi-abelian algebraically coherent theory.
We study equivalent conditions in the context of semi-abelian categories, as well as some of its consequences: including amongst others, strong protomodularity, and normality of Higgins commutators for normal subobjects, and in the varietal case, fibre-wise algebraic cartesian closedness. 
\end{abstract}


\date{\today}

\maketitle

%\tableofcontents

\section{Introduction}
The aim of this article is to study a condition which recently arose in some loosely interrelated categorical-algebraic investigations~\cite{MM, CigoliMontoliCharSubobjects, CMM}: we ask of a finitely complete category that the change-of-base functors of its fibration of points are \emph{coherent}, which means that they preserve finite limits and jointly strongly epimorphic pairs of arrows. 

Despite its apparent simplicity, this property---which we shall call \emph{algebraic coherence}---has some important consequences.
For instance, any algebraically coherent semi-abelian category~\cite{Janelidze-Marki-Tholen} satisfies the so-called \emph{Smith is Huq} condition \SH.
In fact (see Section~\ref{Section Consequences}) it also satisfies the \emph{strong protomodularity} condition as well as the conditions~\SSH, which is a strong version of \SH, and \NH, \emph{normality of Higgins commutators of normal subobjects}---studied in~\cite{Bourn2004,Borceux-Bourn}, \cite{Borceux-Bourn,MFVdL}, \cite{MFVdL3} and~\cite{AlanThesis,CGrayVdL1}, respectively.
Nevertheless, there are many examples including all \emph{categories of interest} in the sense of Orzech~\cite{Orzech} (Theorem~\ref{theorem:ci.jse}).
In particular, the categories of groups, (commutative) rings (not necessarily unitary), Lie algebras over a commutative ring with unit, Poisson algebras and associative algebras are all examples, as well as the categories of (compact) Hausdorff groups, Lie algebras etc.
Knowing that a category is not only semi-abelian, but satisfies these additional conditions is crucial for many results in categorical algebra, in particular to applications in (co)homology theory.
For instance, the description of internal crossed modules~\cite{Janelidze} becomes simpler when \SH\ holds~\cite{MFVdL,HVdL}; the theory of universal central extensions depends on the validity of both \SH\ and \NH~\cite{CVdL,GrayVdL1}; and under \SH\ higher central extensions admit a characterisation in terms of binary commutators which helps in the interpretation of (co)homology groups~\cite{RVdL3,RVdL2}. 

The concept of \emph{algebraically coherent} category is meant to be an algebraic version of the concept of \emph{coherent} category~\cite{Johnstone:Elephant}, as explained by a certain formal parallel between topos theory and categorical algebra~\cite{Janelidze:Topos-talk}.
The key idea is that notions which in topos theory are expressed by properties of the basic fibration $\cod\colon \Arr(\C) \to \C$ may have a meaningful counterpart in categorical algebra when the basic fibration is replaced by the fibration of points $\cod\colon \Pt(\C) \to \C$.
That is to say, the slice categories $(\C\downarrow X)$ are replaced by the categories $\Pt_{X}(\C)$ of points over $X$ in $\C$, whose objects are split epimorphisms with a chosen splitting $(p\colon Y\to X, s\colon X\to Y)$, $ps=1_{X}$.
A successful example of this parallel is the second author's notion of \emph{algebraically cartesian closed} category---see~\cite{Gray2012,Bourn-Gray} and related works.
The present paper provides a new example: while a coherent category is a regular category $\C$ where every change-of-base functor of the basic fibration $\cod\colon \Arr(\C) \to \C$ is coherent, an algebraically coherent category is a finitely complete category $\C$ where the same property holds for the fibration of points $\cod\colon \Pt(\C) \to \C$.
As a consequence, certain results carry over from topos theory to categorical algebra for purely formal reasons: for instance, in parallel with the long-established~\cite[Lemma 1.5.13]{Johnstone:Elephant}, any locally algebraically cartesian closed category is algebraically coherent (Theorem~\ref{(LACC) implies (AC)}).
We shall see that this procedure (replacing the basic fibration with the fibration of points) is indeed necessary, because while a semi-abelian category~\cite{Janelidze-Marki-Tholen} may or may not be algebraically coherent---see Section~\ref{Examples} for a list of examples---it is never coherent, unless it is trivial (Proposition~\ref{proposition: unital coherent trivial}).

Section~\ref{(AC)} recalls the definitions of coherent functor and coherent category.
In Section~\ref{(ACC)} we define algebraic coherence, characterise it in terms of the kernel functor alone (Proposition~\ref{Proposition Characterisation kernel functor}, Theorem~\ref{Theorem Bemol}) and study its stability properties: closure under slices and coslices (Proposition~\ref{Prop Slice-coslice}), points (Corollary~\ref{Corollary Fibres}), and (regular epi)-reflections (Proposition~\ref{Proposition Reflective Subcat}).
In Section~\ref{Examples} we give examples, non-examples and counterexamples.
The main results here are Theorem~\ref{theorem:ci.jse} proving that all \emph{categories of interest} in the sense of Orzech are algebraically coherent, and Theorem~\ref{Topological algebras} which says that categories of (compact) Hausdorff algebras over a semi-abelian algebraically coherent theory are still algebraically coherent.
Section~\ref{Subsection Diamond} treats the relationship with two-nilpotency.
In Section~\ref{Section Consequences} we focus on categorical-algebraic consequences of algebraic coherence, mostly in the semi-abelian context. First of all, any pointed Mal'tsev category which is algebraically coherent is protomodular (Theorem~\ref{theorem: Mal'tsev implies protomodular} and the more general Theorem~\ref{theorem: Mal'tsev coherent proto non pointed}).
Next we show that \SH, \NH, \SSH\ and strong protomodularity are all consequences of algebraic coherence (see Theorems~\ref{(SH) + (NH)}, ~\ref{(SSH)} and ~\ref{(SSH) => (SP)}).
In a general context including all varieties of algebras, algebraic coherence implies fibre-wise algebraic cartesian closedness \FWACC\ (see Theorem~\ref{(FWACC)}), meaning that centralisers exist in the fibres of the fibration of points.
Section~\ref{Ternary commutator} focuses on the higher-order Higgins commutator and a proof of the \emph{Three Subobjects Lemma for normal subobjects} (Theorem~\ref{3SO}).
The final section gives a short summary of results that hold in the semi-abelian context.

\tableofcontents
%\pagebreak

% Definitions, first results and equivalent conditions

\section{Coherent functors, coherent categories}\label{(AC)}
Recall that a cospan $(f,g)$ over an object $Z$ in an arbitrary category is called a
\begin{enumerate}
\item \defn{jointly extremally epimorphic pair}
when for each commutative diagram as on the left in Figure~\ref{jeep}, if $m$ is a monomorphism, then $m$ is an isomorphism;
\item
\defn{jointly strongly epimorphic pair} when for each commutative diagram as on the right in Figure~\ref{jeep}, if $m$ is a monomorphism, then there exists a unique morphism $\varphi\colon Z \to M$ such that $m \varphi =\phi$.
\end{enumerate}

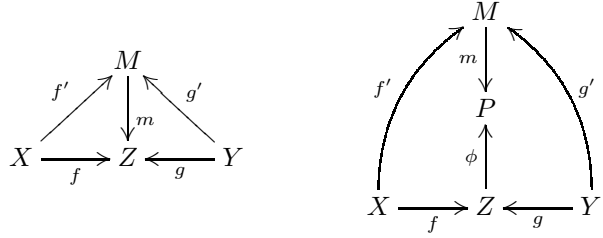
\begin{figure}[h]
$\vcenter{\xymatrix{& M \ar@{->}[d]^(.6){m} \\
X \ar[r]_-{f} \ar[ru]^-{f'} & Z & Y \ar[ul]_-{g'} \ar[l]^-{g}}}
\qquad\qquad
\vcenter{\xymatrix{
& M \ar@{->}[d]_-{m}\\
&P\\
X \ar[r]_-{f} \ar@/^3ex/[ruu]^-{f'} & Z\ar[u]^-{\phi} & Y \ar@/_3ex/[uul]_-{g'} \ar[l]^-{g}}}
$
\caption{Jointly extremally epimorphic and jointly strongly epimorphic pairs}\label{jeep}
\end{figure}
In a similar way to extremal epimorphisms and strong epimorphisms (see for instance Section~1 in~\cite{Janelidze-Sobral-Tholen}) we have

\begin{lemma}
\label{lemma: strong implies extremal}
Let $\C$ be an arbitrary category and let $(f,g)$ be a cospan over an object~$Z$.
If the pair $(f,g)$ is jointly strongly epimorphic, then it is jointly extremally epimorphic.
If $\C$ has pullbacks then $(f,g)$ is jointly extremally epimorphic if and only if it is jointly strongly epimorphic.\noproof
\end{lemma}

\begin{lemma}
\label{lemma: composite of strong/extremal with jointly strong/extremal}
In an arbitrary category, let $(f\colon X\to Z,g\colon Y\to Z)$ be a cospan over~$Z$ and let $e\colon W\to X$ be a strong epimorphism.
\begin{enumerate}
\item $(f,g)$ is jointly extremally epimorphic if and only if $(fe,g)$ is jointly extremally epimorphic;
\item $(f,g)$ is jointly strongly epimorphic if and only if $(fe,g)$ is jointly strongly epimorphic.
\newcounter{temp}
\setcounter{temp}{\value{enumi}}
\end{enumerate}
Suppose now that $(f,g)$ is a jointly strongly epimorphic cospan and consider a morphism $p\colon Z \to V$.
\begin{enumerate}
\setcounter{enumi}{\value{temp}}
\item if $p$ is an extremal epimorphism, then $(pf,pg)$ is jointly extremally epimorphic;
\item if $p$ is a strong epimorphism, then $(pf,pg)$ is jointly strongly epimorphic.\noproof
\end{enumerate}
\end{lemma}

\begin{lemma}
\label{lemma: joins and extremal}
For each commutative diagram
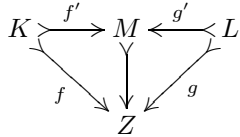
\[
\xymatrix{K \ar@{{ >}->}[r]^-{f'} \ar@{{ >}->}[rd]_{f} & M \ar@{{ >}->}[d] & L \ar@{{ >}->}[ld]^-{g} \ar@{{ >}->}[l]_-{g'}\\
& Z}
\]
in an arbitrary category, $M \leq Z$ is the join of $K\leq Z$ and $L\leq Z$ if and only if $(f',g')$ is jointly extremally epimorphic.
In particular
$(f,g)$ is jointly extremally epimorphic when the diagram above with $M =Z$ is a join.\noproof
\end{lemma}

\begin{lemma}
\label{lemma: jointly becomes single}
For each diagram
\[
\xymatrix{
& X+Y\ar[d]^(0.6){\mus{f}{g}} &\\
X \ar[ur]^{\iota_X} \ar[r]_{f} & Z & Y \ar[ul]_{\iota_Y} \ar[l]^{g}
}
\]
in a category with binary coproducts, $f$ and $g$ are jointly extremally epimorphic / jointly strongly epimorphic if and only if~$\mus{f}{g}$ is an extremal epimorphism / strong epimorphism.\noproof
\end{lemma}

Since in the rest of the paper all categories considered will have finite limits we will freely interchange ``jointly strongly epimorphic'' and ``jointly extremally epimorphic'' (see Lemma~\ref{lemma: strong implies extremal} above).

\begin{definition}\label{Definition Coherent Functor}
A functor between categories with finite limits is called \defn{coherent} if it preserves finite limits and jointly strongly epimorphic pairs.
\end{definition}

Since an arrow is a monomorphism if and only if its kernel pair is the discrete equivalence relation, it follows that any functor which preserves kernel pairs, preserves monomorphisms.
In particular every coherent functor preserves monomorphisms.
Note that in a regular category a morphism $f$ is a regular epimorphism if and only if $(f,f)$ is a jointly strongly epimorphic pair.
It easily follows that a coherent functor between regular categories is always \defn{regular}, that is, it preserves finite limits and regular epimorphisms. 

The next proposition shows that for regular categories, the above definition coincides with the one given in Section~A.1.4 of~\cite{Johnstone:Elephant}.

\begin{proposition}\label{Proposition joins}
A regular functor between regular categories is coherent if and only if it preserves binary joins of subobjects.
\end{proposition}
\begin{proof}
Note that by Lemma~\ref{lemma: composite of strong/extremal with jointly strong/extremal} (b) a cospan $(f,g)$ in a regular category is jointly strongly epimorphic if and only if the cospan $(\Im(f),\Im(g))$ is jointly strongly epimorphic.
Note also that any regular functor preserves (regular epi, mono)-factorisations.
Therefore the proof follows from Lemma~\ref{lemma: joins and extremal}: under either condition diagrams of the form as in Lemma~\ref{lemma: joins and extremal} are preserved.
\end{proof}

\begin{proposition}\label{Proposition binary sums}
Let $F\colon \C \to \D$ be a left exact functor between categories with finite limits and binary coproducts.
The following are equivalent:
\begin{tfae}
\item $F$ is coherent;
\item $F$ preserves strong epimorphisms and the comparison morphism
\[
{\mus{F(\iota_X)}{F(\iota_Y)}\colon F(X)+F(Y)\to F(X+Y)}
\]
is a strong epimorphism for all $X$, $Y\in \C$.
\end{tfae}
When in addition $\C$ is pointed (or more generally when coproduct injections are monomorphisms), these condition are further equivalent to:
\newcounter{tmp}
\setcounter{tmp}{\value{enumi}}
\begin{tfae}
\setcounter{enumi}{\value{tmp}}
\item $F$ preserves strong epimorphisms and binary joins;
\item $F$ preserves strong epimorphisms and joins of the form
\[
\xymatrix{
X \ar[r]^-{\iota_X} & X+Y & Y. \ar[l]_-{\iota_Y}
}
\]
\end{tfae}
\end{proposition}
\begin{proof}
For any jointly strongly epimorphic cospan $(f,g)$ over an object $Z$ consider the diagram
\[
\xymatrix@R=7ex{
& F(X) + F(Y)\ar[d]|{\mus{F(\iota_X)}{F(\iota_Y)}} &\\
& F(X+Y)\ar[d]|{F\mus{f}{g}} &\\
F(X) \ar@/^3ex/@{{ >}->}[uur]^{\iota_{F(X)}} \ar@{{ >}->}[ur]^{F(\iota_X)} \ar@{{ >}->}[r]_{F(f)} & F(Z) & F(Y). \ar@{{ >}->}[ul]_{F(\iota_Y)} \ar@{{ >}->}[l]^{F(g)}\ar@/_3ex/@{{ >}->}[uul]_{\iota_{F(Y)}}
}
\]
Suppose that (ii) holds.
It follows from Lemma~\ref{lemma: jointly becomes single} and the fact that $F$ preserves strong epimorphisms that $F\mus{f}{g}$ is a strong epimorphism.
Therefore the vertical composite $F\mus{f}{g}\mus{F(\iota_X)}{F(\iota_Y)}=\mus{F(f)}{F(g)}$ is a strong epimorphism and so according to Lemma~\ref{lemma: jointly becomes single} the cospan $(F(f),F(g))$ is jointly strongly epimorphic.
This proves that (ii) implies~(i).
Since (ii) and (iii) follow trivially from (i), and (iv) follows from (iii), it remains only to show that (iv) implies (ii).
However this follows from Lemma
\ref{lemma: joins and extremal} and~\ref{lemma: jointly becomes single}. 
\end{proof}

\begin{definition}\label{Definition Coherent Category}
A regular category with finite coproducts $\C$ is \defn{coherent} in the sense of~\cite{Johnstone:Elephant} (and called a \defn{pre-logos} in~\cite{Alligators}) if and only if, for any morphism $f\colon{X\to Y}$ in $\C$, the change-of-base functor $f^{*}\colon (\C\downarrow Y)\to (\C\downarrow X)$ is coherent.
\end{definition}

The categories $\Gp$ and $\Ab$ (all groups, abelian groups) are well-known not to be coherent.
In fact, the only semi-abelian (or, more generally, unital) coherent category is the trivial one.
Recall from~\cite{B0,Borceux-Bourn} that a pointed finitely complete category is \defn{unital} when for any pair of objects $X$, $Y$ the cospan
\[
\xymatrix{X \ar[r]^-{\gauche 1_{X},0\droite} & X\times Y & Y \ar[l]_-{\gauche 0,1_{Y}\droite} }
\]
is jointly strongly epimorphic.

\begin{lemma}\label{lemma: unital diagonal is not coherent}
Let $\C$ be a unital category.
For each object $X$ in $\C$ the pullback functor $\gauche 1_{X},1_{X}\droite^{*} \colon (\C \downarrow (X\times X)) \to (\C \downarrow X)$ is coherent if and only if $X$ is a zero object.
\end{lemma}
\begin{proof}
In the diagram
\[
\xymatrix{
0 \ar[r] \ar[d]& X \ar[d]^-{\gauche 1_{X},1_{X}\droite} & 
0 \ar[d] \ar[l]\\
X \ar[r]_-{\gauche 1_{X},0\droite} & X \times X & X \ar[l]^-{\gauche 0,1_{X}\droite}
}
\]
the two squares are pullbacks and $(\gauche 1_{X},0\droite, \gauche 0,1_{X}\droite)$ is a jointly strongly epimorphic cospan in $\C$, and hence in $(\C\downarrow (X\times X))$.
It follows that $0\to X$ is a strong epimorphism, so that $X$ is isomorphic to $0$. 
\end{proof}

\begin{proposition}\label{proposition: unital coherent trivial}
If a unital category is coherent, then it is trivial.
\end{proposition}
\begin{proof}
The proof follows trivially from Lemma~\ref{lemma: unital diagonal is not coherent}.
\end{proof}

However, we will see that in a unital category certain change-of-base functors are always coherent.

\begin{lemma}
Let $\C$ be a unital category.
If $(f,g)$ and $(f',g')$ are jointly strongly epimorphic cospans over $Z$ and $Z'$ respectively, then $(f\times f',g\times g')$ is a jointly strongly epimorphic cospan over $Z\times Z'$.
\end{lemma}
\begin{proof}
Consider the diagram
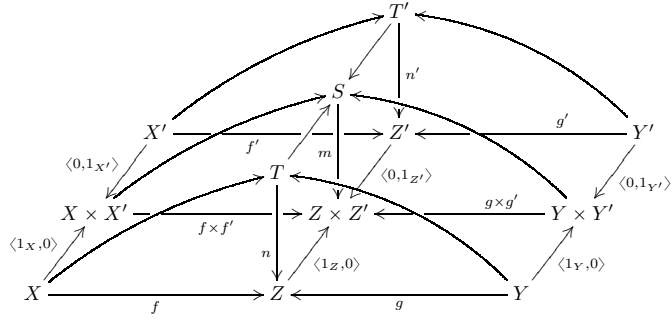
\[
\resizebox{.7\textwidth}{!}{\xymatrix@!0@C=3em@R=2em{
& & & & & & T'\ar[ddd]^{n'}|(.75){\hole} \ar[ddl]\\
& & & & & &\\
& & & & & S\ar[ddd]_{m}|(.75){\hole}\\
& & X' \ar[rrrr]|(.25){\hole}_(0.4){f'}|(.62){\hole}|(.75){\hole} \ar[ddl]_{\gauche 0,1_{X'}\droite}\ar@/^2ex/[rrrruuu] & & & & Z'\ar[ddl]^(.4){\gauche 0,1_{Z'}\droite}|(.75){\hole} & & & & Y' \ar[llll]_(.33){g'}|(.69){\hole} \ar[ddl]^{\gauche 0,1_{Y'}\droite}\ar@/_3ex/[lllluuu]\\
& & & & T\ar[ddd]_(.66){n} \ar[uur] & & & & & &\\ 
& X\times X' \ar[rrrr]|(.25){\hole}_{f\times f'}|(.75){\hole} \ar@/^2ex/[rrrruuu] & & & & Z\times Z' & & & & Y\times Y' \ar[llll]_(.33){g\times g'}|(.67){\hole} \ar@/_3ex/[lllluuu]\\
& & & & & & & & &\\
X \ar[rrrr]_{f}\ar[uur]^{\gauche 1_{X},0\droite}\ar@/^2ex/[rrrruuu] & & & & Z\ar[uur]_{\gauche 1_{Z},0\droite} & & & & Y \ar[llll]^{g}\ar[uur]_{\gauche 1_{Y},0\droite} \ar@/_3ex/[lllluuu]}}
\]
where $m$ monomorphism of cospans and the monomorphisms of cospans $n$ and $n'$ are obtained by pullback.
Since $(f,g)$ and $(f',g')$ are jointly strongly epimorphic cospans it follows that $n$ and $n'$ are isomorphisms, respectively.
Therefore since $\C$ is unital it follows that $m$ is an isomorphism as required. 
\end{proof}

As an immediate corollary we obtain:

\begin{lemma}
\label{lemma: unital product with functor coherent}
Let $\C$ be a unital category.
For each object $X$ in $\C$ the functor ${X\times (-)}\colon \C \to \C$ is coherent, and hence so are the change-of-base functors ${\C \to (\C \downarrow X)}$ and ${\C \to \Pt_{X}(\C)}$ along $X\to 0$.\noproof 
\end{lemma}

\section{Algebraically coherent categories}\label{(ACC)}
Considering that even the most basic algebraic categories are never coherent, it is natural to try and find an algebraic variant of the concept.
The idea followed in this paper is to replace the \emph{basic fibration} by the \emph{fibration of points}:

\begin{definition}\label{Definition Algebraically Coherent Category}
A category with finite limits is called \defn{algebraically coherent} if and only if for each morphism $f\colon{X\to Y}$ in~$\C$, the change-of-base functor
\[
f^{*}\colon {\Pt_{Y}(\C)\to \Pt_{X}(\C)}
\]
is coherent.
\end{definition}

This definition means that for each diagram
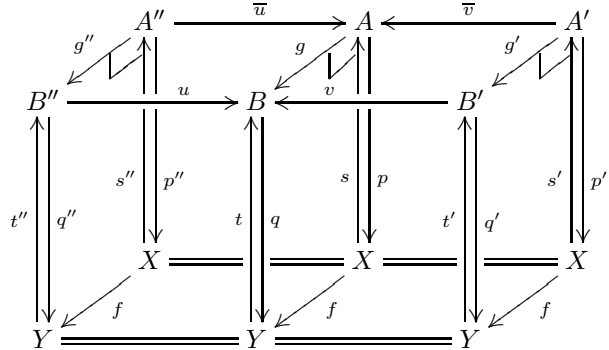
\[
\vcenter{\xymatrix@!0@R=3em@C=4em{
& A'' \skewpullback \ar[rr]^-{\overline u} \ar[dl]_-{g''} \ar@<.5ex>[ddd]|(.33){\hole}^(.66){p''} & & A \skewpullback \ar[dl]_-{g} \ar@<.5ex>[ddd]|(.33){\hole}^(.66){p} & & A' \skewpullback \ar[dl]_-{g'} \ar[ll]_-{\overline v} \ar@<.5ex>[ddd]^(.66){p'} \\
B'' \ar[rr]^(.66){u} \ar@<.5ex>[ddd]^-{q''} & & B \ar@<.5ex>[ddd]^{q} & & B' \ar[ll]_(.66){v} \ar@<.5ex>[ddd]^{q'} &\\
& & & & &\\
& X \ar@{=}[rr]|(.47){\hole}|(.53){\hole} \ar@<.5ex>[uuu]^(.35){s''}|(.67){\hole} \ar[dl]^(.4){f} & & X \ar@<.5ex>[uuu]^(.35){s}|(.67){\hole} \ar@{=}[rr]|(.47){\hole}|(.53){\hole} \ar[dl]^(.4){f} & & X \ar@<.5ex>[uuu]^(.35){s'} \ar[dl]^(.4){f}\\
Y \ar@<.5ex>[uuu]^-{t''} \ar@{=}[rr] & & Y \ar@<.5ex>[uuu]^-t \ar@{=}[rr] & & Y \ar@<.5ex>[uuu]^-{t'}}}
\]
where $(u,v)$ is a cospan in $\Pt_Y(\C)$ and $(\overline u,\overline v)$ is the cospan in $\Pt_X(\C)$ obtained by change-of-base along $f$, if $(u,v)$ is a jointly strongly epimorphic pair, then also the pair $(\overline u,\overline v)$ is jointly strongly epimorphic.
Note that we can interpret those conditions in $\C$ itself:

\begin{lemma}\label{Lemma Subobject}
Each jointly strongly epimorphic pair in a category of points $\Pt_X(\C)$ is still jointly strongly epimorphic when considered in $(\C\downarrow X)$ or even $\C$.
\end{lemma}
\begin{proof}
Consider a jointly strongly epimorphic pair $(u,v)$ in $\Pt_X(\C)$ which factors through a subobject $m$ in $\C$.
\[
\xymatrix{
	& M \ar@{{ >}->}[d]^-m \\
	A'' \ar[ru]^-{\overline u} \ar[r]^-{u} \ar@<.5ex>[d]^-{p''} & A \ar@<.5ex>[d]^{p} & A' \ar[l]_-{v} \ar@<.5ex>[d]^-{p'} \ar[lu]_-{\overline v} \\
	X \ar@{=}[r] \ar@<.5ex>[u]^-{s''} & X \ar@<.5ex>[u]^-{s} & X \ar@{=}[l] \ar@<.5ex>[u]^-{s'}
}
\]
Then, clearly, $pm$ is split by $\overline{u}s''=\overline{v}s'\colon {X\to M}$, thus $m$, $\overline{u}$ and $\overline{v}$ become morphisms of points.
\end{proof}

\subsection{Stability properties.}
Next we will show that if a category is algebraically coherent, then so are its slice and coslice categories and so is any
full subcategory which is closed under products and subobjects.

\begin{proposition}\label{Prop Slice-coslice}
If a category $\C$ is algebraically coherent, then, for each $X$ in~$\C$, the categories $(\C\downarrow X)$ and $(X\downarrow\C)$ are also algebraically coherent.
\end{proposition}

\begin{proof}
For each morphism in the slice category $(\C \downarrow X)$, i.e.~a commutative diagram
$$ \xymatrix{
	Y \ar[dr]_\alpha \ar[rr]^f & & Z \ar[dl]^\beta \\
	& X
} $$
in $\C$, there are isomorphisms of categories (the horizontal arrows below) which make the diagram
$$ \xymatrix{
	\Pt_{(Z,\beta)}(\C\downarrow X) \ar[r]^-\cong \ar[d]_{(f \downarrow X)^{*}} & \Pt_{Z}(\C) \ar[d]^{f^*} \\
	\Pt_{(Y,\alpha)}(\C\downarrow X) \ar[r]_-\cong & \Pt_{Y}(\C)
} $$
commute.
It follows that $(f\downarrow X)^{*}$ is coherent whenever $f^*$ is.
A similar argument holds for the coslice category $(X \downarrow \C)$.
\end{proof}

\begin{corollary}\label{Corollary Fibres}
If a category $\C$ is algebraically coherent, then any fibre $\Pt_{X}(\C)$ is also algebraically coherent.
\end{corollary}
\begin{proof}
Since $\Pt_X(\C)=((X,1_X)\downarrow(\C\downarrow X))$, this follows from Proposition~\ref{Prop Slice-coslice}.
\end{proof}

\begin{proposition}\label{Proposition Functor Categories}
If $\C$ is algebraically coherent, then so is any category of diagrams in $\C$. In particular, such is the category $\Pt(\C)$ of points in $\C$. 
\end{proposition}
\begin{proof}
Since in a functor category, limits and colimits are pointwise, the passage to categories of diagrams in~$\C$ is obvious.
\end{proof}

\begin{proposition}\label{Proposition Reflective Subcat}
If $\B$ is a full subcategory of an algebraically coherent category~$\C$ closed under finite products and subobjects (and hence all finite limits), then $\B$ is algebraically coherent.
In particular, any (regular epi)-reflective subcategory of an algebraically coherent category is algebraically coherent.
\end{proposition}
\begin{proof}
We have to show that, for each morphism $f \colon X \to Y$ in \B, the change-of-base functor $f^*\colon{\Pt_{Y}(\B)\to \Pt_{X}(\B)}$ is coherent.
Since the category $\B$ is closed under finite limits in~$\C$ this functor is a restriction of the change-of-base functor $f^*\colon{\Pt_{Y}(\C)\to \Pt_{X}(\C)}$. It therefore suffices to note that cospans in~$\B$ are jointly strongly epimorphic in~$\B$ if and only if they are in~$\C$.
However since $\B$ is closed under subobjects in $\C$, this is indeed the case. 
\end{proof}

\subsection{The protomodular case.}
We recall that a category $\C$ is called \defn{protomodular} in the sense of Bourn~\cite{Bourn1991} if it has pullbacks of split epimorphisms along any map and all the change-of-base functors of the fibration of points $\cod\colon \Pt(\C)\to \C$ are \defn{conservative}, which means that they reflect isomorphisms. See also~\cite{Borceux-Bourn} for a detailed account of this notion.

It is an obvious consequence of Lemma~\ref{lemma: composite of strong/extremal with jointly strong/extremal} that any change-of-base functor along a pullback-stable strong epimorphism (and in particular along regular epimorphisms in a regular category) \emph{reflects} jointly strongly epimorphic pairs (see also Lemma~\ref{lemma: pulling back along pullback-stable extremal}~(c) below). We now explore the protomodular case, where \emph{all} change-of-base functors reflect jointly strongly epimorphic pairs.
Using this result we will prove that when $\C$ is a pointed protomodular category, algebraic coherence can be expressed in terms of kernel functors $\Ker\colon \Pt_{X}(\C)\to \C$ (which are precisely the change-of-base functors along initial maps $!_{X}\colon{0\to X}$) alone. 

\begin{lemma}\label{Lemma Reg Protomodular}
Let $F\colon \C\to \D$ be a functor. If $F$ is conservative and preserves monomorphisms then it reflects jointly strongly epimorphic pairs.
\end{lemma}
\begin{proof}
Suppose that $(u,v)$ is a cospan in $\C$ such that $(F(u),F(v))$ is a jointly strongly epimorphic pair in $\D$.
This means that the image through $F$ of any monomorphism of cospans with codomain $(u,v)$ in $\C$ is an isomorphism.
The proof now follows from the fact that $F$ reflects isomorphisms.
\end{proof}

\begin{proposition}\label{Proposition Regular Protomodular}
If $\C$ is a protomodular category, then the change-of-base functors reflect jointly strongly epimorphic pairs.\noproof
\end{proposition}

\begin{lemma}\label{lemma: reflect and preserve}
Let $F\colon \C\to \D$ and $G\colon \D\to \E$ be functors.
If $GF$ and $G$ preserve and reflect, respectively, jointly strongly epimorphic pairs, then $F$ preserves jointly strongly epimorphic pairs. 
\end{lemma}
\begin{proof}
Let $(u,v)$ be a jointly strongly epimorphic cospan.
By assumption, it follows that $(GF(u),GF(v))$ and hence $(F(u), F(v))$ is a jointly strongly epimorphic cospan.
\end{proof}

\begin{proposition}\label{Proposition Characterisation kernel functor}
A protomodular category $\C$ with an initial object is algebraically coherent if and only if
the change-of-base functors along each morphism from the initial object are coherent.
In particular a pointed protomodular category is algebraically coherent if and only if the kernel functors $\Ker\colon \Pt_{X}(\C)\to \C$ are coherent.
\end{proposition}
\begin{proof}
Since by Proposition~\ref{Proposition Regular Protomodular} every change-of-base functor reflects jointly strong\-ly epimorphic pairs, the non-trivial implication follows from Lemma~\ref{lemma: reflect and preserve} applied to the commutative triangle
\[
\xymatrix@!0@R=3em@C=4em{\Pt_{Y}(\C) \ar[rr]^-{f^*} \ar[rd]_-{!_Y^*} && \Pt_{X}(\C) \ar[ld]^-{!_X^*}\\
& \Pt_{0}(\C)}
\] 
where $f\colon {X\to Y}$ is an arbitrary morphism in $\C$ and $0$ is the initial object in $\C$.
\end{proof}

It is worth spelling out what Proposition~\ref{Proposition binary sums} means in a pointed protomodular category with pushouts of split monomorphisms.

\begin{proposition}\label{AC through sum}
A pointed protomodular category with pushouts of split monomorphisms is algebraically coherent if and only if for every diagram of split extensions of the form
\begin{equation}\label{Sum in points}
\vcenter{\xymatrix{
	H \ar@{{ |>}->}[d]_-{h} \ar@{{ >}->}[r] & K \ar@{{ |>}->}[d] & L \ar@{{ |>}->}[d]^-{l} \ar@{{ >}->}[l] \\
	A \ar@{{ >}->}[r]^-{\iota_{A}} \ar@<.5ex>[d]^{p'} & A+_{X}C \ar@<.5ex>[d]^{p} & C \ar@{{ >}->}[l]_-{\iota_{C}} \ar@<.5ex>[d]^{p''} \\
	X \ar@{=}[r] \ar@<.5ex>[u]^-{s'} & X \ar@<.5ex>[u]^-{s} & X \ar@{=}[l] \ar@<.5ex>[u]^-{s''}
}}
\end{equation}
the induced arrow $H+L \to K$ is a strong epimorphism.
\end{proposition}
\begin{proof}
This is a combination of Proposition~\ref{Proposition binary sums} (i) $\Leftrightarrow$ (ii) and Proposition~\ref{Proposition Characterisation kernel functor}.
\end{proof}

Using Proposition~\ref{Proposition binary sums}, this result may be rephrased as follows. Note the resemblance with the strong protomodularity condition (cf.\ Subsection~\ref{SP}).

\begin{corollary}\label{for join normal}
A pointed protomodular category with pushouts of split monomorphisms is algebraically coherent if and only if for each diagram such as~\eqref{Sum in points}, $K$~is the join of $H$ and $L$ in~${A+_XC}$.\noproof
\end{corollary}

\subsection{Algebraic coherence in terms of the functors $X\flat (-)$.}\label{flat}
We end this section with a characterisation of algebraic coherence in terms of the action monad $X\flat (-)$.
Recall from~\cite{Bourn-Janelidze:Semidirect,BJK} that $X\flat (-)\colon{\C\to \C}$ takes an object $Y$ and sends it to the kernel in the short exact sequence
\[
\xymatrix{
0 \ar[r] & X\flat Y \ar@{{ |>}->}[r]^-{\kappa_{X,Y}} & X + Y \ar@{-{ >>}}@<0.5ex>[r]^-{\mus{1_{X}}{0}} & X \ar@<0.5ex>[l]^-{\iota_X} \ar[r] & 0.}
\]
This functor is part of a monad on $\C$, induced by the adjunction $(X+(-))\dashv \Ker $, for which the algebras are called \defn{internal $X$-actions} and which gives rise to a comparison $\Pt_{X}(\C)\to \C^{X\flat(-)}$.
For instance, any internal action $\xi\colon{X\flat Y\to Y}$ of a group $X$ on a group $Y$ corresponds to a homomorphism ${X\to \Aut(Y)}$.
In other categories, however, interpretations of the concept of internal object action may look very different~\cite{BJK,BJK2}.

\begin{lemma}\label{Lemma Bemol}
If $\C$ is a pointed algebraically coherent category with binary coproducts, then for any object $X$, the functor $X\flat(-)\colon {\C\to \C}$ preserves jointly strongly epimorphic pairs.
\end{lemma}
\begin{proof}
This follows from the fact that kernel functors are coherent while left adjoints preserve jointly strongly epimorphic pairs.
\end{proof}

\begin{lemma}\label{lemma: covering}
Let $F\colon \C\to \D$ and $G\colon \D\to \E$ be functors such that $\C$ has binary coproducts and $F$ preserves them, $F$ preserves jointly strongly epimorphic pairs, $G$ preserves strong epimorphisms, and for every $D$ in $\D$ there exists a strong epimorphism ${F(C) \to D}$.
$GF$ preserves jointly strongly epimorphic pairs if and only if $G$ does.
\end{lemma}
\begin{proof}
The ``if'' part follows from the fact that the composite of functors which preserve jointly strongly epimorphic pairs, preserves jointly strongly epimorphic pairs.
For the ``only if'' part
let $(g_1,g_2)$ be a jointly strongly epimorphic cospan and construct the diagram
\[
\xymatrix@!0@C=4em@R=2em{
F(C_1) \ar[rr]^-{F(\iota_{C_{1}})} \ar[dd]_{e_1} && F(C_1 + C_2) \ar[dd]^{e} && F(C_2) \ar[ll]_-{F(\iota_{C_{2}})}\ar[dd]^{e_2}\\
&& &&\\
D_1\ar[rr]_{g_1} && D && D_2 \ar[ll]^{g_2}
}
\]
where $e_1$ and $e_2$ are arbitrary strong epimorphism existing by assumption, and $e$ is induced by the coproduct.
Since $(g_1,g_2)$ is jointly strongly epimorphic and $e_1$ and~$e_2$ are strong, $e$ is necessarily strong by Lemmas~\ref{lemma: composite of strong/extremal with jointly strong/extremal} and~\ref{lemma: jointly becomes single}.
Therefore, since $G$ preserves strong epimorphisms and $GF$ preserves jointly strongly epimorphic pairs it follows that
\[
(G(e) GF(\iota_{C_{1}}), G(e) GF(\iota_{C_{2}})) = (G(g_1)G(e_1),G(g_2)G(e_2))
\]
is a jointly strongly epimorphic cospan, and so $(G(g_1),G(g_2))$ is jointly strongly epimorphic by Lemma~\ref{lemma: composite of strong/extremal with jointly strong/extremal}.
\end{proof}

One situation where this lemma applies is when $F\dashv G$ is an adjunction where the functor $G$ preserves strong epimorphisms with a strongly epimorphic counit.
Taking ${(X+(-))\dashv \Ker}$ we find, in particular, Theorem~\ref{Theorem Bemol}. Recall from~\cite{Borceux-Bourn} that a pointed, regular and protomodular category is called \defn{homological}.

\begin{theorem}\label{Theorem Bemol}
Let $\C$ be a homological category with binary coproducts.
$\C$ is algebraically coherent if and only if for every $X$, the functor $X\flat(-)\colon {\C\to \C}$ preserves jointly strongly epimorphic pairs.\noproof
\end{theorem}

In the article~\cite{MM}, the authors consider a variation on this condition, asking that the functors $X\flat(-)$ preserve jointly epimorphic pairs formed by semidirect product injections.

Note that in the proof of Theorem~\ref{Theorem Bemol} we did not use the existence of coequalisers in $\C$, so it is actually valid in any  pointed protomodular category with binary coproducts in which strong epimorphisms are pullback-stable.

Lemma~\ref{Lemma Bemol} can also be used as follows. Recall from~\cite{JMU, Janelidze-Marki-Ursini} that in a pointed and regular category, a \defn{clot} is a subobject $K\leq Y$ such that the conjugation action on $Y$ restricts to it.

\begin{proposition}\label{Prop join actions}
Let $\C$ be a pointed algebraically coherent category with binary coproducts and binary joins of subobjects. Given $K$, $L\leq Y$ in $\C$, if $\xi\colon{X\flat Y\to Y}$ is an action which restricts to $K$ and $L$, then $\xi$ restricts to~${K\join L}$. In particular, if $K$ and $L$ are clots in $Y$, then so is $K \join L$.
\end{proposition}
\begin{proof}
Let us consider the diagram
\[
\xymatrix{
	& & X\flat L \ar[dl]_{X\flat j} \ar[dr]^{X\flat l} \ar@<.5ex>[dd]^(.35){\xi_L}|(.5){\hole}|(.64){\hole} \\
	& X\flat(K\join L) \ar@{-->}@<.5ex>[dd]^(.55){\xi_{K \join L}} \ar[rr]_(.35){X\flat m}
		& & X\flat Y \ar@<.5ex>[dd]^\xi \\
	X \flat K \ar[ur]^{X \flat i} \ar@<.5ex>[dd]^{\xi_K} \ar[urrr]_{X\flat k}
		& & L \ar[dl]_j \ar[dr]^l \ar@<.5ex>[uu]|(.35){\hole}|(.5){\hole}^(.65){\eta_L} \\
	& K\join L \ar@<.5ex>[uu]^(.45){\eta_{K\join L}}|(.67){\hole} \ar@{ >->}[rr]^m & & Y \ar@<.5ex>[uu]^{\eta_Y} \\
	K \ar[ur]^i \ar[urrr]_k \ar@<.5ex>[uu]^{\eta_K}}
\]
where the arrows at the bottom floor are all inclusions of subobjects of $Y$, $\eta$ is the unit of the monad $X\flat(-)$, and $\xi$ is an action of $X$ on $Y$, with restrictions $\xi_K$ and~$\xi_L$ to $K$ and $L$ respectively.

Since, by Lemma~\ref{Lemma Bemol}, the pair $(X\flat i,X\flat j)$ is jointly strongly epimorphic, and $m$ is a monomorphism, there exists a unique $\xi_{K \join L}$ (the dashed arrow in the diagram) such that $\xi_{K \join L}(X \flat i)=i\xi_K$, $\xi_{K \join L}(X \flat j)=j\xi_L$ and $m\xi_{K \join L}=\xi(X \flat m)$.

It is not difficult to see that $\xi_{K \join L}$ is indeed a retraction of $\eta_{K \join L}$. In order to prove that it is an action, and hence an algebra for the monad $X\flat(-)$, we still have to show that the diagram
$$ \xymatrix@C=4em{
	X\flat(X\flat(K\join L)) \ar[d]_{\mu_{K \join L}} \ar[r]^-{X\flat \xi_{K \join L}}
		& X\flat(K\join L) \ar[d]^{\xi_{K \join L}} \\
	X\flat(K\join L) \ar[r]^-{\xi_{K \join L}} & K \join L,
} $$
where $\mu$ is the multiplication of the monad $X\flat(-)$, commutes.
To prove this we use that the analogous property holds for both $\xi_K$ and $\xi_L$ so that, again by Lemma~\ref{Lemma Bemol}, the pair $(X \flat (X\flat i),X\flat(X \flat j))$ is jointly strongly epimorphic.
\end{proof}

Recall that a subobject in a pointed category is called \defn{Bourn-normal} when it is the normalisation of an equivalence relation~\cite[Section 3.2]{Borceux-Bourn}. In an exact homological category, Bourn-normal subobjects and kernels (= normal subobjects) coincide.

\begin{corollary}
In an algebraically coherent homological category with binary coproducts, the join of two Bourn-normal subobjects is Bourn-normal.
\end{corollary}
\begin{proof}
The result follows from the fact that in this context Bourn-normal subobjects coincide with clots~\cite{MM-NC}.
\end{proof}

Notice that, in an exact homological category, the join of two normal subobjects is always normal~\cite[Corollary 4.3.15]{Borceux-Bourn}. 

In fact, in a semi-abelian context, the property in Proposition~\ref{Prop join actions} turns out to be equivalent to algebraic coherence.

\begin{theorem}
Suppose $\C$ is a semi-abelian category. The following are equivalent:
\begin{tfae}
\item $\C$ is algebraically coherent;
\item given $K$, $L\leq Y$ in $\C$, any action $\xi\colon{X\flat Y\to Y}$ which restricts to $K$ and $L$ also restricts to~${K\join L}$.
\end{tfae}
\end{theorem}
\begin{proof}
One implication is Proposition~\ref{Prop join actions}. Conversely, since here the comparison $\Pt_{X}(\C)\to \C^{X\flat(-)}$ is an equivalence, condition (ii) says that any cospan in $\Pt_{X}(\C)$ whose restriction to kernels is $K$, $L\leq Y$ factors through a morphism in $\Pt_{X}(\C)$ whose restriction to kernels is $K\join L\leq Y$. But since the kernel functor $\Ker\colon {\Pt_X(\C) \to \C}$ reflects jointly strongly epimorphic pairs (Lemma~\ref{Lemma Reg Protomodular}) and monomorphisms, it also reflects joins of subobjects. Then it preserves them, hence it is coherent by Proposition~\ref{Proposition binary sums}. The result now follows from Proposition~\ref{Proposition Characterisation kernel functor}.
\end{proof}

%%%%%%%%%%%%%%%%%%%%%%%%%%%%%%%%%%%%%%%%%%%%%%%%%%%

% Examples, non-examples and counterexamples

\section{Examples, non-examples and counterexamples}\label{Examples}

Before treating algebraic examples, let us first consider those given by topos theory.

\begin{proposition}
Any coherent category is algebraically coherent.
\end{proposition}
\begin{proof}
This is an immediate consequence of Lemma~\ref{Lemma Subobject}.
\end{proof}

\begin{examples}\label{Toposes}
This provides us with all elementary toposes as examples (sets, finite sets, sheaves, etc.). \end{examples}

\begin{example}
The dual of the category of pointed sets is semi-abelian~\cite{Bourn:Dual-topos} algebraically coherent.
One way to verify this is by the dual of the condition of Proposition~\ref{AC through sum} in the category $\Set_{*}$.
Given two elements of $A\times_{X}C$, it suffices to check all relevant cases to see that it is still possible to separate them after~$X$ has been collapsed.
The same argument is valid to prove that $\E_{*}^{\op}$ is algebraically coherent when $\E$ is any boolean topos: the existence of complements allows us to express the cokernel of a monomorphism $m\colon {M\to X}$ in $\Pt_{1}(\E)$ as a disjoint union $(X\setminus M)\sqcup 1$.
Indeed, in the diagram in $\E$
\[
\xymatrix{0 \pullback \ar[r] \ar[d] & M \pullback \ar[r] \ar[d] & 1 \ar[d]\\
X\setminus M \ar[r] & X \ar[r] \pushout & X/M \pushout}
\]
each square is simultaneously a pullback and a pushout---see~\cite{Johnstone:Elephant}.
Being given the opposite
\[
\vcenter{\xymatrix{
	(A\setminus X)\sqcup 1 & (B\setminus X)\sqcup 1 \ar[l]_-{\overline{f}} \ar[r]^-{\overline{g}} & (C\setminus X)\sqcup 1 \\
	A \ar@{-{ >>}}[u] \ar@<.5ex>[d]^{s'} & B \ar@{-{ >>}}[u] \ar[l]_-{f} \ar[r]^-{g} \ar@<.5ex>[d]^{s} & C \ar@{-{ >>}}[u] \ar@<.5ex>[d]^{s''} \\
	X \ar@{=}[r] \ar@<.5ex>[u]^-{p'} & X \ar@<.5ex>[u]^-{p} & X \ar@{=}[l] \ar@<.5ex>[u]^-{p''}
}}
\]
of diagram~\eqref{Sum in points}, page~\pageref{Sum in points}, in~$\Pt_{1}(\E)$, we now have to prove that $\overline{f}$ and $\overline{g}$ are jointly (strongly) monomorphic when so are $f$ and $g$.
Being given $b$, $b'\in (B\setminus X)\sqcup 1$ such that $\overline{f}(b)=\overline{f}(b')$ and $\overline{g}(b)=\overline{g}(b')$, we shall see that $b=b'$.
Without loss of generality, as follows from $\E$ being lextensive, we may assume that one of the three cases
\begin{enumerate}
\item $b$, $b'\in B\setminus X$;
\item $b\in B\setminus X$ and $b'\in 1$;
\item $b\in 1$ and $b'\in B\setminus X$
\end{enumerate}
is satisfied, of which only the first leads to further work.
Things are fine if either $\overline{f}(b)=\overline{f}(b')$ or $\overline{g}(b)=\overline{g}(b')$ is outside $1$.
When, however, both $\overline{f}(b)=1$ and $\overline{g}(b)=1$, then $f(b)=p's'f(b)=p's(b)=fps(b)$ (because $\overline{f}(b)=1$ means that $f(b)$ is in $X$) and $g(b)=gps(b)$ (for similar reasons), which proves that $b=ps(b)\in X$.
\end{example}

\begin{example}
The category $\Top$ of topological spaces and continuous maps is not coherent, because it is not even regular.

In fact, $\Top$ is not \emph{algebraically} coherent either, since the change-of-base functors of the fibration of points need not preserve regular epimorphisms (which coincide here with strong epimorphisms = quotient maps).
To see this, let us consider the following variation on Counterexample~2.4.5 in~\cite{Borceux:Cats2}. Let $A$, $B$, $C$ and $D$ be the topological spaces defined as follows.
Their underlying sets are $\{a,b,c,d\}$, $\{l,m,n\}$, $\{x,y,z\}$ and $\{i,j\}$ respectively.
The topologies on $A$ and $B$ are generated by $\{\{a,b\}\}$ and $\{\{l,m\}\}$ while the topologies of $C$ and $D$ are indiscrete.
Let $f\colon {A\to C}$, $s\colon{D\to A}$, $p\colon{C\to D}$ and $g\colon{B\to D}$ be the continuous maps defined by:
\begin{center}
\begin{tabular}{c|ccccc}
 & $a$ & $b$ & $c$ & $d$ \\[.5ex]
\hline
 $f$ & $x$ & $y$ & $y$ & $z$
\end{tabular}
\qquad
\begin{tabular}{c|cc}
 & $i$ & $j$\\
\hline
 $s$ & $d$ & $c$ 
\end{tabular}
\qquad
\begin{tabular}{c|ccc}
 & $x$ & $y$ & $z$ \\
\hline
 $p$ & $i$ & $j$ & $i$ 
\end{tabular}
\qquad
\begin{tabular}{c|ccc}
 & $l$ & $m$ & $n$ \\
\hline
 $g$ & $i$ & $i$ & $i$ 
\end{tabular}
\end{center}
Then $f$ is actually a regular epimorphism ${(pf,s)\to (p,fs)}$ in $\Pt_{D}(\Top)$.
However, its image $g^{*}(f)$ by the change-of-base functor $g^{*}\colon {\Pt_{D}(\Top)\to \Pt_{B}(\Top)}$ is a surjection, but not a regular epimorphism.
Indeed, since $A\times_D B$ and $C\times_D B$ have underlying sets $\{a,d\}\times \{l,m,n\}$ and $\{x,z\}\times \{l,m,n\}$, and topologies generated by $\{\{a\}\times\{l,m,n\},\{a,d\}\times \{l,m\}\}$ and $\{\{x,z\}\times \{l,m\}\}$ respectively, it follows that the set
\[
\{a\}\times \{l,m\}=(g^{*}(f))^{-1}(\{x\}\times\{l,m\})
\]
is an open subset of $A\times_D B$ coming from a non-open subset of $C \times_D B$.
This means that $C\times_{D}B$ does not carry the quotient topology induced by $g^{*}(f)$ and so $g^{*}(f)$ is not a regular epimorphism.
\end{example}

It is well known~\cite[Lemma 1.5.13]{Johnstone:Elephant} that \emph{any finitely cocomplete locally cartesian closed category is coherent}.
We find the following algebraic version of this classical result.
We recall from~\cite{Gray2012, Bourn-Gray} that a finitely complete category $\C$ is said to be \defn{locally algebraically cartesian closed} (satisfies condition \defn{(LACC)}) when, for every $f\colon X \to Y$ in \C, the change-of-base functor $f^*\colon \Pt_{Y}(\C)\to \Pt_{X}(\C)$ is a left adjoint. 

\begin{theorem}\label{(LACC) implies (AC)}
Any locally algebraically cartesian closed category is algebraically coherent.
\end{theorem}
\begin{proof}
This is a consequence of the fact that change-of-base functors always preserve limits and under \LACC, since they are left adjoints, they preserve jointly strongly epimorphic pairs.
\end{proof}

\begin{example}
The category of cocommutative Hopf algebras over a field $K$ of characteristic zero is semi-abelian as explained in~\cite{Kadjo, GKV}.
It is also locally algebraically cartesian closed by Proposition~5.3 in~\cite{Gray2012}, being the category of internal groups in the category of cocommutative coalgebras, which is cartesian closed as shown in~\cite[Theorem~5.3]{Barr-Coalgebras}.
Incidentally, via 4.4 in~\cite{BJK}, the same argument suffices to show that the category $\Hopf$ has representable object actions.
%More generally, the same holds for internal groups in the category of cocommutative comonoids in any so-called \emph{admissible} (symmetric) monoidal closed category, which is cartesian closed by 3.2 in~\cite{Porst-Comonoids}.
\end{example}

\begin{proposition} \label{Prop NatMal}
Any finitely complete naturally Mal'tsev category~\cite{Johnstone:Maltsev} is algebraically coherent.
\end{proposition}
\begin{proof}
If $\C$ is naturally Mal'tsev, then for each object $X$ of $\C$, the category $\Pt_{X}(\C)$ of points over $X$ is naturally Mal'tsev, pointed and finitely complete, hence it is additive by a proposition in~\cite{Johnstone:Maltsev}.
As a consequence, the change-of-base functors $
f^{*}\colon {\Pt_{Y}(\C)\to \Pt_{X}(\C)}$ all preserve binary coproducts and hence are coherent by Proposition~\ref{Proposition binary sums}.
\end{proof}

\begin{examples}\label{linear examples}
The following are algebraically coherent: all abelian categories, all additive categories, all affine categories in the sense of~\cite{Carboni-Affine-Spaces}.
\end{examples}

\begin{proposition}
Let $\C$ be an algebraically coherent Mal'tsev category~\cite{CLP,CPP}. Then, for any $X$ in \C, the category $\Gpd_X(\C)$ of internal groupoids in $\C$ with object of objects~$X$ is algebraically coherent. In particular, the category $\Gp(\C)$ of internal groups in $\C$ is algebraically coherent.
\end{proposition}
\begin{proof}
For any $X$ in \C, $\Gpd_X(\C)$ is a naturally Mal'tsev category by Theorem 2.11.6 in~\cite{Borceux-Bourn}. The result follows from Proposition~\ref{Prop NatMal}.
\end{proof}

Note that some of the results we shall prove in Section~\ref{Section Consequences} apply only to semi-abelian categories, so need not apply to all the examples above. On the other hand, being semi-abelian is not enough for algebraic coherence.

\begin{examples}
Not all semi-abelian (or even strongly semi-abelian) varieties are algebraically coherent.
We list some, together with the consequence of algebraic coherence which they lack: (commutative) loops and digroups (since by the results in~\cite{Borceux-Bourn,Bourn2004,HVdL} they do not satisfy \SH, see Theorem~\ref{(SH) + (NH)} below), non-associative rings (or algebras in general), Jordan algebras (since as explained in~\cite{AlanThesis,CGrayVdL1} they do not satisfy \NH, see Theorem~\ref{(SH) + (NH)}), and Heyting semilattices (which, as explained in~\cite{MFVdL3}, form an arithmetical~\cite{Borceux-Bourn,Pedicchio2} Moore category~\cite{Rodelo:Moore} that does not satisfy \SSH, see Theorem~\ref{(SSH)}).
\end{examples}

In general, (compact) Hausdorff algebras over an algebraically coherent semi-abelian theory are still algebraically coherent.

\begin{theorem}\label{Topological algebras}
Let $\TT$ be a theory such that $\Set^{\TT}$ is an algebraically coherent semi-abelian variety.
Then the homological category $\Haus^\TT$ and the semi-abelian category $\HComp^{\TT}$ are algebraically coherent.
\end{theorem}
\begin{proof}
According to~\cite{Borceux-Clementino} the category $\Haus^\TT$ is homological and $\HComp^\TT$ is semi-abelian. This means that we may use Proposition~\ref{AC through sum} to show their algebraic coherence.
Let us consider a diagram like~\eqref{Sum in points} in $\Haus^\TT$.
Since $\Set^\TT$ is algebraically coherent, we know that the underlying algebras are such that $H\join L=K$. 
Given a subset $S\subset K$ which is open in the final topology on~$K$ induced by $H$, $L\subset K$, we have to prove that $S$ is open in the subspace topology induced by $K\subset A+_{X}C$.
By definition, $H\cap S$ and $L\cap S$ are open in $H$ and in~$L$, respectively.
Since for Hausdorff algebras kernels are closed~\cite[Proposition~26]{Borceux-Clementino}, $(H\cap S)\cup (A\setminus H)$ and $(L\cap S)\cup (C\setminus L)$ are open in $A$ and in $C$, respectively.
Hence $S\cup ((A+_{X}C) \setminus K)$ is open in $A+_{X}C$, which carries the final topology.

Since limits in $\HComp$ are computed again as in $\Top$, this proof also works for compact Hausdorff algebras. 
\end{proof}
It would be interesting to know whether or not the same result holds for topological spaces.

To make full use of this result, we need further examples of algebraically coherent semi-abelian varieties of algebras.
One class of such are the \emph{categories of interest} in the sense of~\cite{Orzech}.

%Further examples of algebraically coherent semi-abelian varieties of algebras are all \emph{categories of interest} in the sense of~\cite{Orzech}.

\begin{definition}\label{cat.int}
A \defn{category of interest} is a variety of universal algebras whose theory contains a unique constant $0$, a set $\Omega$ of finitary operations and a set of identities $\EE$ such that:
\begin{enumerate}
	\item[\CI1] $\Omega=\Omega_0\cup\Omega_1\cup\Omega_2$, where $\Omega_i$ is the set of $i$-ary operations;
	\item[\CI2] $\Omega_0=\{0\}$, $-\in\Omega_1$ and $+\in\Omega_2$, where $\Omega_i$ is the set of $i$-ary operations, and $\EE$ includes the group laws for $0$, $-$, $+$; define $\Omega_1'=\Omega_1\setminus\{-\}$, $\Omega_2'=\Omega_2\setminus\{+\}$;
	\item[\CI3] for any $*\in\Omega_2'$, the set $\Omega_2'$ contains $*^{\op}$ defined by $x*^{\op} y=y*x$;
	\item[\CI4] for any $\omega\in\Omega_1'$, $\EE$ includes the identity $\omega(x+y)=\omega(x)+\omega(y)$;
	\item[\CI5] for any $*\in\Omega_2'$, $\EE$ includes the identity $x*(y+z)=x*y+x*z$;
	\item[\CI6] for any $\omega\in\Omega_1'$ and $*\in\Omega_2'$, $\EE$ includes the identity $\omega(x)*y=\omega(x*y)$;
	\item[\CI7] for any $*\in\Omega_2'$, $\EE$ includes the identity $x+(y*z)=(y*z)+x$;
	\item[\CI8] for any $*$, $\divideontimes\in\Omega_2'$, there exists a word $w$ such that $\EE$ includes the identity
		$$ (x*y)\divideontimes z=w(x*_1(y\divideontimes_1 z),\ldots, x*_m(y\divideontimes_m z),y*_{m+1}(x\divideontimes_{m+1} z),\ldots,y*_n(x\divideontimes_n z)) $$
		where $*_1,\ldots,*_n$ and $\divideontimes_1,\ldots,\divideontimes_n$ are operations in $\Omega_2'$.
\end{enumerate}
\end{definition}

The following lemma expresses the well-known equivalence between split epimorphisms and actions~\cite{Bourn-Janelidze:Semidirect,BJK} in the special case of a category of interest: here an internal $B$-action on an object is determined by a set of additional operations~$u_{b,*}$, one for each element $b$ of $B$ and each binary operation $*$. 

\begin{lemma}
\label{lemma:cat.int.decomp}
Let $\C$ be a variety of universal algebras whose theory contains a unique constant $0$, a set of finitary operations $\Omega$, and a set of identities $\EE$ such that \CI1--\CI5 of Definition~\ref{cat.int} hold.
For every $B$ in $\C$ define $\C_B$ to be a new variety whose theory contains a unique constant $0$, a set of finitary operations $\Omega_B$, and a set of identities $\EE_B$ such that:
\begin{enumerate}
\item $\Omega_B = \Omega_{B_0}\cup \Omega_{B_1} \cup \Omega_{B_2}$, where $\Omega_{B_i}$ is the set of $i$-ary operations;
\item $\Omega_{B_0} = \Omega_0$, $\Omega_{B_2}=\Omega_2$ and $\Omega_{B_1}= \Omega_1 \sqcup \Theta_1$ where $\Theta_1=\{u_{b,*}\mid b\in B, * \in \Omega_2\}$;
\item $\EE_B$ has the same identities as in $\EE$ but in addition for each $u_{b,*}$ in $\Theta_1$ the identity $u_{b,*}(x+y)=u_{b,*}(x) + u_{b,*}(y)$. 
\end{enumerate}
The functor $I_B\colon \Pt_B(\C)\to\C_B$ sending a split epimorphism
\[
\xymatrix{
A \ar@<0.5ex>@{-{ >>}}[r]^{\alpha} & B \ar@<0.5ex>[l]^{\beta}
}
\]
to the kernel of $\alpha$ with all operations induced by those on $A$ except for the unary operations $u_{b,*}$ which are defined by
\[
u_{b,*}(x) = \begin{cases}
\beta(b) + x - \beta(b) & \qquad \text{if $*=+$}\\
\beta(b)*x & \qquad \text{otherwise} 
\end{cases}
\]
is such that $\overline{\C_B} = I_B(\Pt_B(\C))$ is a subvariety of $\C_B$ and $I_{B}\colon {\Pt_{B}(\C)\to \overline{\C_B}}$ is an equivalence of categories. 

Moreover if conditions \CI6--\CI8 of Definition~\ref{cat.int} also hold, then for every $n$-ary word $w$ of $\overline{\C_B}$ there exists an $m$-ary word $w'$ of $\C$ and unary words $v_{i,1}$, $v_{i,2}$, \dots, $v_{i,{m_i}}$ of $\overline{\C_B}$ for each $i$ in $\{1,\dots,n\}$ such that
\begin{align*}
w(x_1,\dots,x_n) &= w'(v_{1,1}(x_1),\dots,v_{1,m_1}(x_1),\\
 &\phantom{= w'(}v_{2,1}(x_2),\dots,v_{2,m_2}(x_2),\dots,v_{n,1}(x_n),\dots,v_{n,m_n}(x_n)).
\end{align*} 
\end{lemma}
\begin{proof}
For a semi-abelian category, kernel functors are always faithful, since they preserve equalisers and reflect isomorphisms.
Hence the functor $I_B$ is faithful too, because the kernel functor factors through it.
Since the kernel functor reflects limits, being conservative by protomodularity, it follows that $I_B$ does too.
This proves that $\Pt_B(\C)$ is closed under limits in $\C_B$.

For each $X$ in $\C_B$ we can define all operations in $\Omega$ on the set $X\times B$ as follows:
\begin{align*}
0 &= (0,0) &&\text{is the unique constant}\\
u(x,b) &= (u(x),u(b)) && \text{for each $u$ in $\Omega'_1$}\\
-(x,b) &= (u_{-b,+}(-x),-b)\\
(x,b) + (y,c) &= (x+u_{b,+}(y),b+c) &&\\
(x,b)*(y,c) &= (x*y + u_{b,*}(y) + u_{c,*^{\op}}(x), b*c) && \text{for each $*$ in $\Omega'_2$.} 
\end{align*}
The set $X\times B$ equipped with these operations becomes an object of \C, and the maps $\pi_2 \colon X\times B \to B$ and $\gauche 0,1_{B}\droite \colon {B\to X\times B}$ are morphisms in \C.
If 
\[
\xymatrix{
X=I_B(A \ar@{-{ >>}}@<0.5ex>[r]^-{\alpha} & B \ar@<0.5ex>[l]^-{\beta})
}
\]
then the map $\varphi \colon X\times B \to A$ defined by $\varphi(x,b) = x + \beta(b)$ is a bijection which preserves all operations.
Indeed
\begin{align*}
\varphi(u(x,b)) &= \varphi(u(x),u(b))
= u(x) + \beta(u(b))
= u(\varphi(x,b))\\
&\\
\varphi((x,b)+(y,c)) &= x+u_{b,+}(y) + \beta(b+c)\\
&= x + \beta(b) + y- \beta(b) + \beta(b) + \beta(c)\\
&=\varphi(x,b)+ \varphi(y,c)
\end{align*}
\begin{align*}
\varphi((x,b)*(y,c))&= x*y + u_{b,*}(y) + u_{c,*^{\op}}(x) + \beta(b*c)\\
&=x*y + \beta(b)*y + x*\beta(c) + \beta(b)*\beta(c)\\
&=(x+\beta(b))*(y+\beta(c))\\
&=\varphi(x,b)*\varphi(y,c).
\end{align*}
Next we will show that for each $f\colon X\to X'$ in $\C_B$ the map $f\times 1_{B } \colon X\times B \to X'\times B$ which trivially makes the diagram
\[
\xymatrix{
X \ar[r]^-{\gauche 1_{X},0\droite} \ar[d]_-{f} & X\times B \ar@<0.5ex>[r]^-{\pi_2}\ar[d]^-{f\times 1_{B}} & B\ar@{=}[d] \ar@<0.5ex>[l]^-{\gauche 0,1_{B}\droite}\\
X' \ar[r]_-{\gauche 1_{X'},0\droite} & X'\times B \ar@<0.5ex>[r]^-{\pi_2} & B \ar@<0.5ex>[l]^-{\gauche 0,1_{B}\droite}
}
\]
commute also preserves the operations defined above.
We have
\begin{align*}
(f\times 1_{B})(u(x,b)) &= (f\times 1_{B})(u(x),u(b))
=(f(u(x)),u(b))\\
&=u((f\times 1_{B})(x,b))
\end{align*}
\begin{align*}
(f\times 1_{B})((x,b)+(y,c))&= (f\times 1_{B})(x + u_{b,+}(y),b+c)\\
&=(f(x+u_{b,+}(y)),b+c)\\
&=(f(x),b)+(f(y),c)\\
&=(f\times 1_{B})(x,b)+(f\times 1_{B})(y,c)
\end{align*}
\begin{align*}
(f\times 1_{B})((x,b)*(y,c)) &= (f\times 1_{B})(x*y+u_{b,*}(y) + u_{c,*^{\op}}(x),b*c)\\
&=(f(x*y + u_{b,*}(y) + u_{c,*^{\op}}(x)),b*c)\\
&=(f(x),b)*(f(y),c)\\
&=(f\times 1_{B})(x,b)*(f\times 1_{B})(y,c).
\end{align*}

This means that $I_B$ is full, and also that $\Pt_B(\C)$ is closed under monomorphisms and quotients in $\C_B$.
Indeed, $f\times 1_{B}$ is a monomorphism or a regular epimorphism as soon as $f$ is. 

It is easy to check that 
\begin{align*}
0+x&=x && \text{using \CI2} \\
0*x&=0 && \text{when $*\neq +$ using \CI5}\\
-(x+y) &= (-y)+ (-x) && \text{using \CI2}\\
-(x*y) &= (-x)*y && \text{when $*\neq+$ using \CI2, \CI2 and \CI5}
\end{align*}
and for each $u$ in $\Omega'_1$
\begin{align*}
u(x+y) &= u(x) + u(y) && \text{using \CI4}\\
u(x*y) &= u(x)*y && \text{using \CI6}
\end{align*}
which means that for each $n$-ary word $w$ from $\C$ there exists an $n$-ary word $w'$ built using only operations from $\Omega_2$, and unary words $v_1$, \dots, $v_n$ which are composites of operations from $\Omega_1$ such that $w(x_1,\dots,x_n) = w'(v_1(x_1),\dots,v_n(x_n))$. 
It is also easy to check that for each $u_{b,*}$ in $\Theta_1$
\begin{align*}
u_{b,*}(x+y) &= u_{b,*}(x) + u_{b,*}(y) && \text{using \CI2 for $*=+$ and \CI5 otherwise}\\
u_{b,+}(x\divideontimes y) &= x\divideontimes y && \text{when $\divideontimes\neq +$ using $\CI2$, $\CI7$.}
\end{align*}
When $*\neq +$ and $\divideontimes \neq +$, according to \CI3 and \CI8 and what was proved above, there exists a word $w$ built using only operations from $\Omega_2$ and unary words $v_1$, \dots, $v_n$ which are composites of operations from $\Omega_1$ such that
\begin{align*}
u_{b,*}(x\divideontimes y) = w\bigl(&v_1(x\divideontimes_1(u_{b,*_1}(y))),\dots,v_{m}(x\divideontimes_m(u_{b,*_m}(y))),\\
&v_{m+1}(y \divideontimes_{m+1} (u_{b,*_{m+1}}(x))),\dots,v_n(y_n\divideontimes_n(u_{b,*_n}(x)))\bigr)\\
= w\bigl(&x\divideontimes_1(u_{b,*_1}(v_1(y))),\dots,x\divideontimes_m(u_{b,*_m}(v_{m}(y))),\\
&y \divideontimes_{m+1} (u_{b,*_{m+1}}(v_{m+1}(x))),\dots,y_n\divideontimes_n(u_{b,*_n}(v_n(x)))\bigr).
\end{align*}
The final claim follows by induction.
\end{proof}

\begin{lemma}\label{lemma:nice forgetful functors}
Let $U \colon \B \to \C$ be a forgetful functor between varieties (meaning that the operations and identities of $\C$ are amongst those of $\B$) such that for each $n$-ary word $w$ in $\B$ there exists an $m$-ary word $w'$ in $\C$ and unary words $v_{i,1}, v_{i,2}, \dots,v_{i,{m_i}}$ in $\B$ for each $i\in \{1,\dots,n\}$ satisfying
\begin{align*}
w(x_1,\dots,x_n) = w'(&v_{1,1}(x_1),\dots,v_{1,m_1}(x_1),\\
 &v_{2,1}(x_2),\dots,v_{2,m_2}(x_2),\dots,v_{n,1}(x_n),\dots,v_{n,m_n}(x_n)).
\end{align*} 
The functor $U$ is coherent.
\end{lemma}
\begin{proof}
Since every element of $U(X+Y)$ is of the form $a=w(x_1,\dots,x_k,y_{k+1}\dots,y_n)$ for some $n$-ary word $w$ from $\B$, where $x_1$, \dots, $x_k$ are in $X$ and $y_{k+1}$, \dots, $y_n$ are in~$Y$, it follows by assumption that there exist a word $w'$ from $\C$ and $v_{i,1}$, $v_{i,2}$, \dots, $v_{i,{m_i}}$ for each $i$ in $\{1,\dots,n\}$ in $\B$ such that $w(x_1,\dots,x_k,y_{k+1},\dots,y_n)$ equals
\begin{align*}
w'(&v_{1,1}(x_1),\dots,v_{1,m_1}(x_1), v_{2,1}(x_2),\dots,v_{2,m_2}(x_2),\dots,v_{k,1}(x_k),\dots,v_{k,m_k}(x_k),\\
	&v_{k+1,1}(y_{k+1}),\dots,v_{k+1,m_{k+1}}(y_{k+1}),\dots,v_{n,1}(y_n),\dots,v_{n,m_n}(y_n)).
\end{align*} 
Therefore, since each $v_{i,m_i}(x_i)$ is in $X$ and each $v_{i,m_i}(y_i)$ is in $Y$ it follows that $a$ is in the image of
\[
\mus{U(\iota_X)}{U(\iota_Y)}\colon U(X)+U(Y) \to U(X+Y)
\]
and so $U$ is coherent by Proposition~\ref{Proposition binary sums}. 
\end{proof}

\begin{theorem}\label{theorem:ci.jse}
Every \emph{category of interest} in the sense of Orzech is algebraically coherent.
\end{theorem}
\begin{proof}
The proof is a consequence of Lemma~\ref{lemma:cat.int.decomp} together with Lemma~\ref{lemma:nice forgetful functors} because any kernel functor ${\Pt_{B}(\C)\to \C}$ factors into an equivalence $I_{B}\colon {\Pt_{B}(\C)\to \overline{\C_{B}}}$ followed by a coherent functor $U\colon{\overline{\C_{B}}\to \C}$.
\end{proof}

\begin{examples}
The categories of groups and non-unital (Boolean) rings are algebraically coherent semi-abelian categories, as are the categories of associative algebras, Lie algebras, Leibniz algebras, Poisson algebras over a commutative ring with unit, all \emph{varieties of groups} in the sense of~\cite{Neumann}, and all categories of compact Hausdorff such algebras.
\end{examples}

\begin{proposition}
If $\C$ is a semi-abelian algebraically coherent category and $X$ is an object of $\C$, then the category $\Act_{X}(\C)=\C^{X\flat (-)}$ of $X$-actions in $\C$ is semi-abelian algebraically coherent. 
\end{proposition}
\begin{proof}
This is an immediate consequence of Corollary~\ref{Corollary Fibres}, using the equivalence between actions and points from~\cite{Bourn-Janelidze:Semidirect,BJK}, see also Subsection~\ref{flat}.
\end{proof}

\begin{proposition}\label{Proposition XMod}
If $\C$ is algebraically coherent, then the category $\RG(\C)$ of reflexive graphs in $\C$ is algebraically coherent.

If, moreover, $\C$ is exact Mal'tsev, then also the category $\Cat(\C)$ of internal categories (=~internal groupoids) in $\C$ is algebraically coherent.
As a consequence, the category $\Eq(\C)$ of (effective) equivalence relations in $\C$ is algebraically coherent.

If, moreover, $\C$ is semi-abelian then, by equivalence, the categories $\PXMod(\C)$ and $\XMod(\C)$ of (pre)crossed modules in $\C$ are algebraically coherent.
\end{proposition}
\begin{proof}
The first statement follows from Proposition~\ref{Proposition Functor Categories}. We now assume that $\C$ is exact Mal'tsev.
Since the category of internal categories of~$\C$ is (regular epi)-reflective in $\RG(\C)$, we have that $\Cat(\C)$ is algebraically coherent by Proposition~\ref{Proposition Reflective Subcat}.
In turn, following~\cite{Gran:Central-Extensions, Bourn-comprehensive}, we see that the category $\Eq(\C)$ is (regular epi)-reflective in $\Cat(\C)$.
The final claim in the semi-abelian context now follows from the results of~\cite{Janelidze}.
\end{proof}

\begin{examples}
Crossed modules (of groups, rings, Lie algebras, etc.); $n$-cat-groups, for all $n$~\cite{Loday}; groups in a coherent category.
\end{examples}

\begin{proposition}
If $\C$ is an algebraically coherent exact Mal'tsev category, then 
\begin{enumerate}
\item the category $\Arr(\C)$ of arrows in $\C$,
\item its full subcategory $\Ext(\C)$ determined by the extensions (=~regular epimorphisms), and 
\item the category $\CExt_{\B}(\C)$ of $\B$-central extensions~\cite{Janelidze-Kelly} in $\C$, for any Birkhoff subcategory $\B$ of $\C$,
\end{enumerate}
are all algebraically coherent.
\end{proposition}
\begin{proof}
(a) follows from Proposition~\ref{Proposition Functor Categories} since $\Arr(\C)$ is a category of diagrams in~$\C$. (b) follows from Proposition~\ref{Proposition XMod}, because $\Ext(\C)$ and $\Eq(\C)$ are equivalent categories.
(c) now follows from~(b) by Proposition~\ref{Proposition Reflective Subcat}.
\end{proof}

\begin{examples}
Inclusions of normal subgroups (considered as a full subcategory of $\Arr(\Gp)$); central extensions of groups, Lie algebras, crossed modules, etc.; discrete fibrations of internal categories (considered as a full subcategory of $\Arr(\Cat(\C))$) in an algebraically coherent semi-abelian category $\C$~\cite[Theorem~3.2]{Gran:Central-Extensions}.
\end{examples}

\begin{proposition}
Any sub-quasivariety (in particular, any subvariety) of an algebraically coherent variety is algebraically coherent.
\end{proposition}
\begin{proof}
Since any sub-quasivariety is a (regular epi)-reflective subcategory~\cite{Maltsev}, this follows from Proposition~\ref{Proposition Reflective Subcat}.
\end{proof}

\begin{examples}
$n$-nilpotent or $n$-solvable groups, rings, Lie algebras etc.; torsion-free (abelian) groups, reduced rings.
\end{examples}

\subsection{Monoids.}
We end this section with some partial algebraic coherence properties for monoids.

\begin{proposition}\label{Prop Monoids}
If $X$ is a monoid satisfying the quasi-identity $xy=1 \Rightarrow yx=1$, then the kernel functor $\Ker\colon{\Pt_{X}(\Mon)\to \Mon}$ is coherent.
\end{proposition}
\begin{proof}
Given a diagram
\[
\vcenter{\xymatrix{
	H \ar@{{ |>}->}[d]_-{h} \ar@{{ >}->}[r] & K \ar@{{ |>}->}[d] & L \ar@{{ |>}->}[d]^-{l} \ar@{{ >}->}[l] \\
	A \ar@{{ >}->}[r]^-{i_{A}} \ar@<.5ex>[d]^{p'} & B \ar@<.5ex>[d]^{p} & C \ar@{{ >}->}[l]_-{i_{C}} \ar@<.5ex>[d]^{p''} \\
	X \ar@{=}[r] \ar@<.5ex>[u]^-{s'} & X \ar@<.5ex>[u]^-{s} & X \ar@{=}[l] \ar@<.5ex>[u]^-{s''}
}}
\]
where $B=A\join C$, we consider $A$ and $C$ as subsets of $B$ via the monomorphisms $i_{A}$ and $i_{C}$.
We need to show that any element $k$ of~$K$ written as a product $k=a_{1}c_{1}\cdots a_{n}c_{n}$ of elements of $A$ and $C$ in $B$ may be written as a product of elements of $H$ and $L$ in $K$.
We prove this by induction on the length of the product $a_{1}c_{1}\cdots a_{n}c_{n}$. 

When $k=ac$, first note that since $p(ac)=1$ it follows that $p''(c)p'(a)=1$ and so $s'(p''(c))s''(p'(a))=1$.
Hence
\[
k=ac=a s'(p''(c))\cdot s''(p'(a))c,
\]
where $a s'(p''(c))\in H$ and $s''(p'(a))c\in L$.

If $k=a_{1}c_{1}a_{2}\cdots a_{n}c_{n}$, then $p(a_{1}c_{1}a_{2}\cdots a_{n})=p''(c_{n})^{-1}$.
Hence 
\begin{align*}
k&=a_{1}c_{1}a_{2}\cdots a_{n}c_{n}\\
&=a_{1}c_{1}a_{2}\cdots a_{n}s'(p''(c_{n}))\cdot s''(p(a_{1}c_{1}a_{2}\cdots a_{n}))c_{n}
\end{align*}
is a product of two elements of $K$, where the first has length $n-1$ and the second is in $L$.
\end{proof}

As a consequence, both the category $\MonC$ of monoids with cancellation and the category $\CMon$ of commutative monoids have coherent kernel functors.

\begin{remark}
Although we shall not explore this further here, it is worth noting that the category of \emph{all} monoids is \defn{relatively} algebraically coherent: if we replace the fibration of points in Definition~\ref{Definition Algebraically Coherent Category} by the \emph{fibration of Schreier points} considered in~\cite{SchreierBook}, all kernel functors $\Ker\colon{\SPt_{X}(\Mon)\to \Mon}$ will be coherent.
To see this, it suffices to modify the proof of Proposition~\ref{Prop Monoids} as follows.

If $k=ac$, use~\cite[Lemma~2.1.6]{SchreierBook} to write $a$ as $hx$ with $h\in H$ and $x\in X$.
Then $k=h\cdot xc$ where $1=p(k)=p(h)\cdot p(xc)=p(xc)$, so that $xc\in L$. 

If $k=a_{1}c_{1}a_{2}\cdots a_{n}c_{n}$, write $a_{1}$ as $hx$ with $h\in H$ and $x\in X$.
Then $k=h\cdot (xc_{1})a_{2}\cdots a_{n}c_{n}$ where $1=p(k)=p(h)\cdot p((xc_{1})a_{2}\cdots a_{n}c_{n})=p((xc_{1})a_{2}\cdots a_{n}c_{n})$.
Hence the induction hypothesis may be used on the product $(xc_{1})a_{2}\cdots a_{n}c_{n}$. 
\end{remark}

\section{The functor $X\cosmash(-)$ and two-nilpotency}\label{Subsection Diamond}
Consider a cospan $(k\colon K\to X, l\colon L\to X)$ in $\C$.
Following~\cite{MM-NC}, we compute the \defn{Higgins commutator} $[K,L]\leq X$ as in the commutative diagram
\begin{equation}\label{diag:higgins}
\begin{aligned}
\xymatrix{
 0 \ar[r] & K \cosmash L \ar@{{ |>}->}[r]^{\iota_{K,L}} \ar@{-{ >>}}[d] & K + L \ar@{-{ >>}}[r]^{\sigma_{K,L}} \ar[d]^-{\mus{k}{l}}
 & K \times L \ar[r] & 0\\
 & [K,L] \ar@{ >->}[r] & X
}
\end{aligned}
\end{equation}
where $\sigma_{K,L}$ is the canonical morphism from the coproduct to the product, $\iota_{K,L}$ is its kernel and $[K,L]$ is the regular image of $\iota_{K,L}$ through $\mus{k}{l}$.

In contrast with Lemma~\ref{Lemma Bemol}, even in a semi-abelian algebraically coherent category $\C$, the \defn{co-smash product} functors $X\cosmash(-)\colon{\C\to \C}$ for $X\in\C$ introduced in~\cite{Smash,MM-NC} need not preserve jointly strongly epimorphic pairs in general.
Indeed, this would imply that Higgins commutators in $\C$ distribute over joins. To see this, observe the following commutative diagram
\[
\xymatrix{
	K \cosmash L \ar@{-{ >>}}[d] \ar[r] & K \cosmash (L \join M) \ar@{-{ >>}}[d] & K \cosmash M \ar@{-{ >>}}[d] \ar[l] \\
	[K,L] \ar@{ >->}[r] & [K,L\join M] & [K,M] \ar@{ >->}[l]}
\]
in which $K$, $L$ and $M$ are all subobjects of a given object $X$. Since the middle vertical arrow in it is a regular epimorphism, saying that the upper cospan is jointly strongly epimorphic would imply that also the bottom cospan is jointly strongly epimorphic, so that $[K, L \join M] = [K,L] \join [K,M]$.
But this property fails in $\Gp$, as the following example shows.

\begin{example}
Let us consider the symmetric group $X=S_4$ and its subgroups
\[
K=\langle(12)\rangle\,, \qquad L=\langle(23)\rangle\,\qquad\text{and} \qquad M=\langle(34)\rangle\,.
\]
Then $L \vee M =\langle(23),(34)\rangle$, $[K,L]=\langle(123)\rangle$, and $[K,M]=0$, while $[K, L \vee M]$ is the alternating group $A_4$.
That is:
\[
[K, L \vee M] \neq [K,L] \vee [K,M]\,.
\]
\end{example}

On the other hand, for a semi-abelian category, this condition on the functors $X\cosmash(-)$ \emph{does} imply algebraic coherence.
Proposition~2.7 in~\cite{Actions} gives us a split short exact sequence
\[
\xymatrix@C=5em{
0 \ar[r] & X\cosmash Y \ar@{{ |>}->}[r]^-{j_{X,Y}} & X\flat Y \ar@{-{ >>}}@<0.5ex>[r]^-{\mus{0}{1_{Y}}\kappa_{X,Y}} & Y \ar@{{ >}->}@<0.5ex>[l]^-{\eta_{Y}} \ar[r] & 0}
\]
so that, for any $X$, $Y\in \C$, the object $X\flat Y$ decomposes as $(X\cosmash Y)\join Y$.
As a consequence, if $X\cosmash(-)$ preserves jointly strongly epimorphic pairs, then so does the functor $X\flat (-)$; algebraic coherence of~$\C$ now follows from Theorem~\ref{Theorem Bemol}.

If a semi-abelian category $\C$ is \defn{two-nilpotent}---which means~\cite{HVdL} that every \emph{ternary} co-smash product $X\cosmash Y\cosmash Z$, which may be obtained as the kernel in the short exact sequence
\[
\xymatrix@=3em{ 0 \ar[r] & X\cosmash Y\cosmash Z \ar@{{ |>}->}[r] &
 (X+Y)\cosmash Z \ar@{-{ >>}}[r] & (X\cosmash Z)\times(Y\cosmash Z) \ar[r] & 0,}
 \]
 is trivial---then Higgins commutators in $\C$ do also distribute over finite joins by Proposition~2.22 of~\cite{HVdL}, see also~\cite{Actions}.
Hence it follows that all functors of the form ${X\cosmash(-)}\colon{\C\to \C}$ preserve jointly strongly epimorphic pairs. Indeed, if we have a jointly strongly epimorphic cospan
$$ \xymatrix{L \ar[r]^f & K & M \ar[l]_g} $$
and denote by $f(L)$ and $g(M)$ the regular images in $K$ of $L$ and $M$, respectively, then $K=f(L)\join g(M)$. Now since $X\cosmash(-)$ preserves regular epimorphisms~\cite[Lemma~5.11]{MM-NC}, if (being a Higgins commutator) it distributes over binary joins, then also the pair $(X\cosmash f,X\cosmash g)$ is jointly strongly epimorphic.

One example of this situation is the category $\Nil_{2}(\Gp)$ of groups of nilpotency class at most $2$.
More generally, this happens in the \defn{two-nilpotent core} $\Nil_{2}(\C)$ of any semi-abelian category $\C$, which is the Birkhoff subcategory of~$\C$ determined by the \defn{two-nilpotent objects}: those $X$ for which $[X,X,X]=0$ where, given three subobjects $K$, $L$, $M\leq X$ represented by monomorphisms $k$, $l$ and~$m$, the \defn{ternary commutator} $[K,L,M]\leq X$ is the image of the composite
 \[
 \xymatrix@=3em{K\cosmash L\cosmash M \ar@{{ |>}->}[r]^-{\iota_{K,L,M}} & K+L+M \ar[r]^-{\muss{k}{l}{m}} & X.}
 \]
 Thus we proved:

 \begin{theorem}\label{Theorem Two-Nilpotent}
 Any two-nilpotent semi-abelian category is algebraically coherent.\noproof
 \end{theorem}

 In any semi-abelian category $\C$, the Huq commutator $[K,L]_{X}$ of two subobjects $K$, $L\leq X$ is the normal closure of the Higgins commutator $[K,L]$ (see Proposition~5.7 in~\cite{MM-NC}), so by by Proposition~4.14 of~\cite{Actions} it may be obtained as the join $[K,L]\join [[K,L],X]\normal X$.
We see that if $\C$ is two-nilpotent, then Huq commutators distribute over joins of subobjects.
Hence if it is, moreover, a variety, it is \emph{algebraically cartesian closed} \ACC\ by~\cite{Gray2014}.
We will, however, prove a stronger result for categories which are merely algebraically coherent: see Theorem~\ref{(FWACC)} below.

 \begin{examples}
 $\Nil_{2}(\C)$ for any semi-abelian category $\C$; modules over a square ring~\cite{BHP}.
 \end{examples}

% Categorical-algebraic consequences

\section{Categorical-algebraic consequences}\label{Section Consequences}

\subsection{Protomodularity.}
We begin this section by showing that a pointed Mal'tsev category which is algebraically coherent is necessarily protomodular---a straightforward generalisation of Theorem 3.10 in~\cite{B9}.

\begin{theorem}\label{theorem: Mal'tsev implies protomodular}
Let $\C$ be a pointed algebraically coherent category.
If $\C$ is a Mal'tsev category, then it is protomodular.
\end{theorem}
\begin{proof}
Let
\[
\xymatrix{
X \ar[r]^{\kappa} & A \ar@<0.5ex>[r]^{\alpha} & B \ar@<0.5ex>[l]^{\beta}
}
\]
be an arbitrary split extension.
Since the diagram
\[
\xymatrix{
A \ar@<0.5ex>[d]^{\alpha} & B\times A \ar[r]^-{1_{B}\times \alpha} \ar[l]_-{\pi_2} \ar@<0.5ex>[d]^{\alpha\pi_2} & B\times B \ar@<0.5ex>[d]^{\pi_2}\\
B\ar@{=}[r]\ar@<0.5ex>[u]^{\beta} & B \ar@{=}[r]\ar@<0.5ex>[u]^{\gauche 1_{B},\beta\droite} & B \ar@<0.5ex>[u]^{\gauche 1_{B},1_{B}\droite}
}
\]
is a product and $\Pt_{B}(\C)$ is a unital category (since $\C$ is Mal'tsev, see~\cite{B0}), it follows that the morphisms 
\[
\xymatrix{
A \ar[r]^-{\gauche \alpha,1_{A}\droite}\ar@<0.5ex>[d]^{\alpha} & B\times A\ar@<0.5ex>[d]^{\alpha\pi_2} & B\times B \ar[l]_-{1_{B}\times \beta}\ar@<0.5ex>[d]^{\pi_2}\\
B\ar@{=}[r]\ar@<0.5ex>[u]^{\beta} & B \ar@{=}[r]\ar@<0.5ex>[u]^{\gauche 1_{B},\beta\droite} & B \ar@<0.5ex>[u]^{\gauche 1_{B},1_{B}\droite}
}
\]
are jointly strongly epimorphic in $\Pt_B(\C)$.
Hence Lemma~\ref{Lemma Subobject} implies that they are jointly strongly epimorphic in $\C$.
Therefore, since in the diagram
\[
\xymatrix{
X \ar[r]^{\kappa}\ar[d]_{\kappa} & A \ar@<0.5ex>[r]^{\alpha}\ar[d]^{\gauche 0,1_{A}\droite} & B \ar@<0.5ex>[l]^{\beta}\ar[d]^{\gauche 0,1_{B}\droite}\\
A \ar[r]^-{\gauche \alpha,1_{A}\droite} \ar@<0.5ex>[d]^{\alpha} & B\times A \ar@<0.5ex>[r]^{1_{B}\times \alpha}\ar@<0.5ex>[d]^{\pi_1} & B\times B \ar@<0.5ex>[l]^{1_{B}\times \beta}\ar@<0.5ex>[d]^{\pi_1}\\
B\ar@{=}[r]\ar@<0.5ex>[u]^{\beta} & B \ar@{=}[r]\ar@<0.5ex>[u]^{\gauche 1_{B},\beta\droite} & B \ar@<0.5ex>[u]^{\gauche 1_{B},1_{B}\droite}
}
\]
the top split extension is obtained by applying the kernel functor to the bottom split extension in $\Pt_B(\C)$, it follows that $\kappa$ and $\beta$ are jointly strongly epimorphic.
Hence $\C$ is protomodular by Proposition~11 in~\cite{Bourn1991}.
\end{proof}

\begin{lemma}\label{lemma: split monos}
Let $\C$ be an arbitrary category with pullbacks.
If $s\colon D \to B$ is a split monomorphism and $\Pt_D(\C)$ is protomodular, then $s^{*} \colon \Pt_{B}(\C) \to \Pt_{D}(\C)$ reflects isomorphisms.
\end{lemma}
\begin{proof}
Again by Proposition~11 in~\cite{Bourn1991}, it is sufficient to show that for each split pullback 
\[
\xymatrix{
C \ar[r]^{r}\ar@<0.5ex>[d]^{\gamma} & A \ar@<0.5ex>[d]^{\alpha}\\
D \ar@<0.5ex>[u]^{\delta} \ar[r]_{s} & B\ar@<0.5ex>[u]^{\beta}
}
\]
the morphisms $r$ and $\beta$ are jointly strongly epimorphic.
However this is an immediate consequence of Lemma~\ref{Lemma Subobject}, because if $f$ is a splitting of $s$, then the morphism~$r$ in the diagram
\[
\xymatrix{
C \ar[r]^{r}\ar@<0.5ex>[d]^{\gamma} & A \ar@<0.5ex>[r]^{\alpha}\ar@<0.5ex>[d]^{f\alpha} & B \ar@<0.5ex>[l]^{\beta}\ar@<0.5ex>[d]^{f}\\
D\ar@<0.5ex>[u]^{\delta} \ar@{=}[r] & D \ar@{=}[r]\ar@<0.5ex>[u]^{\beta s} & D\ar@<0.5ex>[u]^{s}
}
\]
is the kernel of $\alpha$ in $\Pt_D(\C)$.
\end{proof}

\begin{lemma}
\label{lemma: pulling back along pullback-stable extremal}
Let $\C$ be an arbitrary category with pullbacks and let $q\colon D \to B$ be a pullback-stable strong epimorphism.
Then the functor $q^{*}\colon (\C \downarrow B) \to (\C \downarrow D)$ and hence the functor $q^{*} \colon \Pt_B(\C) \to \Pt_D(\C)$ reflects: 
\begin{enumerate}
\item isomorphisms; 
\item monomorphisms; 
\item jointly strongly epimorphic cospans.
\end{enumerate}
\end{lemma}
\begin{proof}
Point (a) is well known~\cite[Proposition~1.6]{Janelidze-Sobral-Tholen}, and depends on the fact that the functors $q^{*}\colon (\C \downarrow B) \to (\C \downarrow D)$ are right adjoints. (b) follows immediately from (a), since $q^{*}$ preserves limits, and monomorphisms are precisely the arrows whose kernel pair projections are isomorphisms. (c) follows from Lemma~\ref{Lemma Reg Protomodular}.
\end{proof}

Recall that, in a category $\C$ with a terminal object $1$, an object $D$ \defn{has global support} when the unique morphism ${D\to 1}$ is a pullback-stable strong epimorphism.
We write $\Inh(\C)$ for the full subcategory of $\C$ determined by the objects with global support.

\begin{lemma}\label{lemma: pullback functors split mono}
Let $\C$ be a category with a terminal object.
Let $D$ be an object with global support for which $\Pt_D(\C)$ is protomodular.
For each morphism $q \colon D\to B$ the pullback functor $q^{*}\colon {\Pt_B(\C) \to \Pt_D(\C)}$ reflects isomorphisms. 
\end{lemma}
\begin{proof}
Let $q\colon D \to B$ be a morphism in $\C$ such that $D\to 1$ is a pullback-stable strong epimorphism, and $\Pt_D(\C)$ is protomodular.
The result follows from Lemma~\ref{lemma: split monos}
and~\ref{lemma: pulling back along pullback-stable extremal} since $q$ can be factored as in the diagram
\[
\xymatrix{
D \ar[rr]^{q} \ar[dr]_{\gauche 1_{D},q\droite} && B\\
& D\times B \ar[ur]_{\pi_2} &
}
\]
where $\gauche 1_{D},q\droite$ is a split monomorphism and $\pi_2$, being a pullback of $D \to 1$, is a pullback-stable strong epimorphism. 
\end{proof}

We obtain a generalisation of Theorem 3.11 in~\cite{B9}:

\begin{theorem}
\label{theorem: Mal'tsev coherent proto non pointed}
Let $\C$ be a Mal'tsev category such that, for any $X\in\C$, $\Pt_X(\C)$ is algebraically coherent.
Then the category $\Inh(\C)$ is protomodular.
In particular, if every object in~$\C$ admits a pullback-stable strong epimorphism to the terminal object, then $\C$ is protomodular.
\end{theorem}
\begin{proof}
Recall~\cite[Example 2.2.15]{Borceux-Bourn} that if $\C$ is a Mal'tsev category, then for any~$X$ in~$\C$ the category $\Pt_X(\C)$ is also Mal'tsev. The proof now follows from Theorem~\ref{theorem: Mal'tsev implies protomodular} and Lemma~\ref{lemma: pullback functors split mono}.
\end{proof}

\begin{remark}
The above theorem together with Corollary~\ref{Corollary Fibres} implies that any algebraically coherent Mal'tsev category which has  an initial object with global support is protomodular.
\end{remark}

\subsection{Higgins commutators, normal subobjects and normal closures.}
We now describe the effect of coherent functors on Higgins commutators (see Subsection~\ref{Subsection Diamond}), normal subobjects and normal closures.

\begin{proposition}\label{prop:pres.hig}
Let $F\colon \C \to \D$ be a coherent functor between regular pointed categories with binary coproducts.
Then $F$ preserves Higgins commutators of arbitrary cospans.
\end{proposition}
\begin{proof}
Consider a cospan $(k\colon K\to X, l\colon L\to X)$ in~$\C$ and the induced diagram~\eqref{diag:higgins} of page~\pageref{diag:higgins}.
Since $F$ is coherent, it preserves finite limits and the comparison morphism ${F(K)+F(L) \to F(K+L)}$ is a regular epimorphism.
Hence, the leftmost vertical arrow in the diagram
$$ \xymatrix@C=10ex{
	F(K) \cosmash F(L) \pullback \ar@{ |>->}[r]^{\iota_{F(K),F(L)}} \ar@{-{ >>}}[d]
		& F(K)+F(L) \ar@{-{ >>}}[d] \ar[r]^{\sigma_{F(K),F(L)}}
		& F(K) \times F(L) \ar[d]^-{\cong} \\
	F(K \cosmash L) \ar@{ |>->}[r]_{F(\iota_{K,L})} & F(K+L) \ar[r]_{F(\sigma_{K,L})}
		& F(K \times L)
} $$
is a regular epimorphism.
Finally, applying $F$ to Diagram (\ref{diag:higgins}) and pasting with the left hand square above, we obtain the square
$$ \xymatrix@C=10ex{
	F(K) \cosmash F(L) \ar@{ |>->}[r]^{\iota_{F(K),F(L)}} \ar@{-{ >>}}[d]
		& F(K)+F(L) \ar[d]^{\mus{F(k)}{F(l)}} \\
	F([K,L]) \ar@{ >->}[r] & F(X)
} $$
showing us that $F([K,L]) \cong [F(K),F(L)]$.
\end{proof}

As shown in~\cite{CigoliMontoliCharSubobjects}, this implies that the derived subobject $[X,X]$ of an object $X$ in an algebraically coherent semi-abelian category is always characteristic.
Recall from~\cite{Actions,MM-NC} that, for any subobject $K\leq X$ in a semi-abelian category, its normal closure in $X$ may be obtained as the join $K\join [K,X]$.

\begin{corollary}\label{Corollary normal closure semi}
Any coherent functor between semi-abelian categories preserves normal closures.
\end{corollary}
\begin{proof}
This is Proposition~\ref{Proposition joins} combined with Proposition~\ref{prop:pres.hig}.
\end{proof}

\subsection{Ideal-determined categories.}
This result can be proved in a more general context which, for instance, includes all ideal-determined categories~\cite{Janelidze-Marki-Tholen-Ursini}. Working towards Theorem~\ref{thm summary}, we first prove some preliminary results. We start with Lemma~\ref{lemma: normal closure} which is a general version of~\cite[Lemma~4.10]{MFVdL3}.

We assume that $\C$ is a pointed finitely complete and finitely cocomplete category. Recall the following list of basic properties~\cite{Janelidze-Marki-Tholen, Janelidze-Marki-Tholen-Ursini, MM-NC}.
\begin{itemize}
\item [(A1)] $\C$ has pullback-stable (normal-epi, mono) factorisations;
\item [(A2)] in $\C$, regular images of kernels are kernels;
\item [(A3)] Hofmann's axiom~\cite{Janelidze-Marki-Tholen}.
\end{itemize}
The category $\C$ is semi-abelian if and only if (A1), (A2) and (A3) hold. When~$\C$ satisfies just (A1) and (A2) it is called \defn{ideal-determined}~\cite{Janelidze-Marki-Tholen-Ursini}. In it is called \defn{normal} when it satisfies (A1). Indeed, it is well-known and easily seen that this happens precisely when $\C$ is regular and regular epimorphisms and normal epimorphisms coincide in $\C$, which is the original definition given in~\cite{ZJanelidze-Snake}.

\begin{lemma}
\label{lemma: normal closure}
Let $\C$ be pointed, finitely cocomplete and regular, satisfying \Atwo.
For a monomorphism $m\colon M\to X$, the monomorphism $\overline{m}\colon M_X\to X$ in the diagram
\[
\xymatrix{
0 \ar[r] & X\flat M \ar@{{ |>}->}[r]^-{\kappa_{X,M}}\ar@{-{ >>}}[d]_{\theta} & X + M \ar@{-{ >>}}@<0.5ex>[r]^-{\mus{1_{X}}{0}}\ar@{-{ >>}}[d]^{\mus{1_{X}}{m}} & X \ar@<0.5ex>[l]^-{\iota_X} \ar[r] & 0\\
& M_{X} \ar@{{ |>}->}[r]_-{\overline{m}} & X,
}
\]
where $\overline m \theta$ is the factorisation of $\mus{1_{X}}{m}\kappa_{X,M}$ as a regular epimorphism followed by a monomorphism, is the normal closure of $m$.
\end{lemma}
\begin{proof}
First note that $\overline{m}$ may be obtained as the image of $\kappa_{X,M}$ along $\mus{1_{X}}{m}$. This monomorphism is normal by \Atwo. Let $\eta_M \colon M\to X\flat M$ be the unique morphism such that
\[
\kappa_{X,M} \eta_M = \iota_M \colon M \to X+M.
\]
$m$ factors through~$\overline{m}$ because $\overline{m} \theta \eta_M = \mus{1_{X}}{m} \kappa_{X,M}\eta_M = \mus{1_{X}}{m}\iota_M=m$. We thus have to show that $\overline{m}$ is the \emph{smallest} normal monomorphism through which $m$ factors.

Let $k\colon K\to X$ be a normal monomorphism and let $f\colon X \to Y$ be a morphism such that $k$ is the kernel of $f$.
Consider the diagram
\[
\xymatrix{
K \ar[d]_-{\eta_K}\ar[dr]^{\iota_K}\\
X\flat K \ar@{{ |>}->}[r]^-{\kappa_{X,K}} \ar[d]_-{\varphi} & X + K \ar@<0.5ex>[r]^-{\mus{1_{X}}{0}}\ar[d]_-{\left\links \begin{smallmatrix}1_{X} & k\\ 1_{X} &0 \end{smallmatrix}\right\rechts} & X \ar@<0.5ex>[l]^-{\iota_X}\ar@{=}[d]\\
K \ar@{{ |>}->}[r]_-{\gauche k,0\droite} \ar@{=}[d] & X\times_Y X \ar@<0.5ex>[r]^-{\pi_2} \ar[d]_{\pi_1} & X \ar@<0.5ex>[l]^-{\links 1_{X},1_{X}\rechts}\ar[d]^{f}\\
K \ar@{{ |>}->}[r]_{k} & X \ar[r]_{f} & Y
}
\]
where
\begin{itemize}
\item $\eta_K$ is the unique morphism making the triangle at the top commute;
\item the bottom right square is a pullback;
\item $\gauche k,0\droite$ is the kernel of $\pi_2$;
\item $\varphi$ is the unique morphism making the top a morphism of split extensions.
\end{itemize}
Since $\gauche k,0\droite$ is a monomorphism it follows that $\varphi \eta_K = 1_K$ and so in the commutative diagram
\[
\xymatrix{
X\flat K \ar[r]^{\kappa_{X,K}} \ar[d]_{\varphi} & X + K \ar[d]^{\mus{1_{X}}{k}}\\
K \ar[r]_{k} & X
}
\]
$k\varphi$ is the (regular epi, mono)-factorisation of $\mus{1_{X}}{k}\kappa_{X,K}$. Now suppose that there exists $t\colon M\to K$ such that $kt=m$. Since there exists a unique morphism $X\flat t\colon X\flat M \to X\flat K$ making the diagram
\[
\xymatrix{
X\flat M \ar[r]^-{\kappa_{X,M}}\ar[d]_{X\flat t} & X + M \ar@<0.5ex>[r]^-{\mus{1_{X}}{0}}\ar[d]_{1_{X}+t} & X \ar@<0.5ex>[l]^-{\iota_X}\ar@{=}[d]\\
X\flat K \ar[r]_-{\kappa_{X,K}} & X + K \ar@<0.5ex>[r]^-{\mus{1_{X}}{0}} & X \ar@<0.5ex>[l]^-{\iota_X}
}
\]
a morphism of split extensions, functoriality of regular images tells us that $\overline{m}$ factors through $k$
\[
\xymatrix@R=3ex@C=3ex{
& X\flat M \ar[dl]_{X\flat t} \ar[rr]^{\kappa_{X,K}} \ar@{-{ >>}}[dd]|{\hole}^(0.3){\theta}& & X+M \ar[dl]_{1_{X}+t}\ar[dd]^{\mus{1_{X}}{m}}\\
X\flat K \ar[rr]_(0.7){\kappa_{X,K}} \ar@{-{ >>}}[dd]_{\varphi} && X + K \ar[dd]^(0.3){\mus{1_{X}}{k}}\\
& M_X \ar@{{ |>}->}[rr]_(0.3){\overline{m}}|(.47){\hole}\ar@{.>}[dl] && X\ar@{=}[dl]\\
K \ar@{{ |>}->}[rr]_{k} && X
}
\]
as required.
\end{proof}

\begin{proposition}\label{proposition: coherent preserves normal closure}
Let $F$ be a functor between pointed regular categories with finite coproducts satisfying \Atwo.
If $F$ is coherent, then $F$ preserves normal closures.
\end{proposition}
\begin{proof}
Let $m \colon M \to X$ be a monomorphism.
Consider the diagram
\[
\xymatrix@C=8ex{
F(X)\flat F(M) \ar[r]^-{\kappa_{F(X),F(M)}}\ar@{.>}[d]_{h} & F(X) + F(M) \ar@<0.5ex>[r]^-{\mus{1_{F(X)}}{0}} \ar@{.>}[d]_-{\mus{F(\iota_X)}{F(\iota_M)}} & F(X) \ar@<0.5ex>[l]^-{\iota_{F(X)}}\ar@{=}[d]\\
F(X\flat M) \ar[r]_-{F(\kappa_{X,M})}\ar[d]_{F(\theta)} & F(X + M) \ar@<0.5ex>[r]^-{F\mus{1_{X}}{0}}\ar[d]^{F\mus{1_{X}}{m}} & F(X) \ar@<0.5ex>[l]^-{F(\iota_X)}\\
F(M_X) \ar[r]_{F(\overline{m})} & F(X),
}
\]
where
\begin{itemize}
\item $\overline{m}\theta$ is the (regular epi, mono)-factorisation of $\mus{1_{X}}{m}\kappa_{X,M}$;
\item $h$ is the unique morphism making the upper part of the diagram into a morphism of split extensions---which exists since $F$ preserves limits.
\end{itemize}
Lemma~\ref{lemma: normal closure} tells us that $\overline{m}$ is the normal closure of $m$.
Since $F$ is coherent it follows by Proposition~\ref{Proposition binary sums} that the dotted middle arrow is a regular epimorphism.
Hence~$h$ is a regular epimorphism, because the top left square is a pullback.
Since~$F$ preserves (regular epi, mono)-factorisations it follows that $F(\overline{m}) (F(\theta) h)$ is the (regular epi, mono)-factorisation of
\[
F\mus{1_{X}}{m} \mus{F(\iota_X)}{F(\iota_M)} \kappa_{F(X),F(M)} = \mus{1_{F(X)}}{F(m)} \kappa_{F(X),F(M)}
\] 
and so, again by Lemma~\ref{lemma: normal closure}, $F(\overline{m})$ is the normal closure of $F(m)$.
\end{proof}

Recall that a functor between homological categories is said to be \defn{sequentially exact} if it preserves short exact sequences.

\begin{corollary}\label{Corollary cokernels}
Any regular functor which preserves normal closures and normal epimorphisms preserves all cokernels.
\end{corollary}
\begin{proof}
It suffices to preserve cokernels of arbitrary monomorphisms, which are in fact the cokernels of their normal closures.
Those are preserved since the functor under consideration is sequentially exact, because it preserves finite limits and normal epimorphisms. 
\end{proof}

Recall from~\cite{MM-NC} that for a pair of subobjects in a normal unital category with binary coproducts, their Huq commutator is the normal closure of the Higgins commutator.
Thus we find:

\begin{corollary}\label{Corollary Huq}
Let $F$ be a coherent functor between normal unital categories with binary coproducts.
If $F$ preserves normal closures, then $F$ preserves Huq commutators of arbitrary cospans.
\end{corollary}
\begin{proof}
This follows from Proposition~\ref{prop:pres.hig}.
\end{proof}

\begin{theorem}\label{thm summary}
Let $\C$ be an algebraically coherent regular category with pushouts of split monomorphisms.
For any morphism $f\colon X \to Y$, consider the change-of-base functor ${f^*\colon \Pt_{Y}(\C)\to \Pt_{X}(\C)}$.
Then
\begin{enumerate}
	\item $f^*$ preserves Higgins commutators of arbitrary cospans.
\end{enumerate}
If, in addition, $\C$ is an ideal-determined Mal'tsev category, then
\begin{enumerate}
	\item[(b)] $f^*$ preserves normal closures;
	\item[(c)] $f^*$ preserves all cokernels;
	\item[(d)] $f^*$ preserves Huq commutators of arbitrary cospans.
\end{enumerate}
In particular, {\rm(a)--(d)} hold when $\C$ is semi-abelian and algebraically coherent.
\end{theorem}
\begin{proof}
Apply the previous results to the coherent functor $f^*$, again using that $\Pt_{X}(\C)$ is unital when $\C$ is Mal'tsev~\cite{B0}.
In particular, (a), (b), (c) and (d) follow from Proposition~\ref{prop:pres.hig}, Proposition~\ref{proposition: coherent preserves normal closure}, Corollary~\ref{Corollary cokernels} and Corollary~\ref{Corollary Huq}, respectively. 
\end{proof}

\subsection{The conditions \SH\ and \NH}
Let us recall (from~\cite{MFVdL}, for instance) that a pointed Mal'tsev category satisfies the \defn{Smith is Huq} condition~\SH\ when two equivalence relations on a given object always centralise each other (= commute in the Smith sense~\cite{Pedicchio, Smith}) as soon as their normalisations commute in the Huq sense~\cite{BG, Huq}. A semi-abelian category satisfies the condition~\NH\ of \defn{normality of Higgins commutators of normal subobjects}~\cite{AlanThesis, CGrayVdL1} when the Higgins commutator of two normal subobjects of a given object is again a normal subobject, so that it coincides with the Huq commutator. Condition~(d) in Theorem~\ref{thm summary} combined with Theorem~6.5 in~\cite{CGrayVdL1} now gives us the following result.

\begin{theorem}\label{(SH) + (NH)}
Any algebraically coherent semi-abelian category satisfies both the conditions~\SH\ and \NH.\noproof
\end{theorem}

\subsection{Peri-abelian categories and the condition \UCE}\label{Peri-abelian}
Recall that a semi-abelian category $\C$ is \defn{peri-abelian} when for all $f\colon X \to Y$, the change-of-base functor ${f^*\colon \Pt_{Y}(\C)\to \Pt_{X}(\C)}$ commutes with abelianisation. Originally established by Bourn in~\cite{Bourn-Peri} as a convenient condition for the study of certain aspects of cohomology, it was further analysed in~\cite{GrayVdL1} where it is shown to imply the \defn{universal central extension condition} \UCE\ introduced in~\cite{CVdL}. As explained in that paper, the condition~\UCE\ is what is needed for the characteristic properties of universal central extensions of groups to extend to the context of semi-abelian categories.

Via Theorem~\ref{(SH) + (NH)} above, Proposition~2.5 in~\cite{CGrayVdL1} implies that all algebraically coherent semi-abelian categories are peri-abelian and thus by~\cite[Theorem~3.12]{CGrayVdL1} satisfy the universal central extension condition \UCE.

\subsection{A strong version of \SH}
In the article~\cite{MFVdL3}, the authors consider a strong version of the \emph{Smith is Huq} condition, asking that the kernel functors
\[
\Ker\colon \Pt_{X}(\C)\to \C
\]
reflect Huq commutativity of arbitrary cospans (rather than just pairs of normal subobjects).
We write this condition \SSH.
Of course \SSH\ $\To$ \SH.
On the other hand, as shown in~\cite{MFVdL3}, \SSH\ is implied by \LACC.
This can be seen as a consequence of Theorem~\ref{(LACC) implies (AC)} in combination with the following result. 

\begin{theorem}\label{(SSH)}
If $\C$ is an algebraically coherent semi-abelian category, then the kernel functors $\Ker\colon \Pt_{X}(\C)\to \C$ reflect Huq commutators.
Hence the category $\C$ satisfies~\SSH.
\end{theorem}
\begin{proof}
We may combine (d) in Theorem~\ref{thm summary} with Lemma~6.4 in~\cite{CGrayVdL1}.
We find precisely the definition of \SSH\ as given in~\cite{MFVdL3}.
\end{proof}

\subsection{Strong protomodularity.}\label{SP}
A pointed protomodular category $\C$ is said to be \defn{strongly protomodular}~\cite{Bourn2004,Borceux-Bourn,Rodelo:Moore} when for all $f\colon X \to Y$, the change-of-base functor ${f^*\colon \Pt_{Y}(\C)\to \Pt_{X}(\C)}$ reflects Bourn-normal monomorphisms. This is equivalent to asking that, for every morphism of split extensions
\[
\xymatrix{
N \ar@{{ |>}->}[r] \ar@{{ >}->}[d]_-{n} & D \ar@<0.5ex>[r] \ar[d]  & X \ar@{=}[d] \ar@<0.5ex>[l] \\
K \ar@{{ |>}->}[r]_-{k} & B \ar@<0.5ex>[r] & X, \ar@<0.5ex>[l]
}
\]
if $n$ is a normal monomorphism then so is the composite $kn$. Theorem~7.3
in~\cite{Bourn2004} says that any finitely cocomplete strongly protomodular homological category satisfies the condition~\SH.

\begin{lemma}\label{Lemma Split extensions}
Let $\C$ and $\D$ be pointed categories with finite limits such that normal closures of monomorphisms exist in $\C$, and let $F\colon \C\to \D$ be a conservative functor.
If $F$ preserves normal closures, then $F$ reflects normal monomorphisms.
\end{lemma}
\begin{proof}
Let $m\colon M\to X$ be a morphism such that $F(m)$ is normal.
Using that~$F$ preserves limits and reflects isomorphisms, is is easily seen that $m$ is a monomorphism.
Now let $n\colon N\to X$ be the
normal closure of $m$ and $i\colon M\to N$ the unique factorisation $m = ni$.
The monomorphism $F(m)$ being normal, we see that $F(i)$ is an isomorphism: $F(i)$ is the unique factorisation of $F(m)$ through its normal closure~$F(n)$.
Since $F$ reflects isomorphisms, $i$ is an isomorphism, and $m$ is normal.
\end{proof}

\begin{theorem}\label{(SSH) => (SP)}
Any algebraically coherent semi-abelian category is strongly protomodular.
\end{theorem}
\begin{proof}
Via Theorem~\ref{thm summary} and the fact that in a semi-abelian category, Bourn-normal monomorphisms and kernels coincide, this follows from Lemma~\ref{Lemma Split extensions}.
\end{proof}

Notice that, in the case of varieties, the last result can also be seen as a consequence of Proposition 9 in~\cite{Borceux-SP}.

\subsection{Fibrewise algebraic cartesian closedness.}
A finitely complete category $\C$ is called \defn{fibrewise algebraically cartesian closed} (shortly \FWACC) in~\cite{Bourn-Gray} if, for each $X$ in \C, $\Pt_X(\C)$ is algebraically cartesian closed. In the same paper, the authors showed that a pointed Mal'tsev category is \FWACC\ if and only if each fibre of the fibration of points has centralisers.

\begin{lemma}\label{Lemma Coproducts}
Let $F\colon \C \to \D$ and $G\colon \D\to \C$ be functors between categories with finite limits and binary coproducts such that $GF=1_\C$ and $G$ reflects isomorphisms.
If $F$ and $G$ are coherent, then $F$ preserves binary coproducts. 
\end{lemma}
\begin{proof}
Since $F$ is coherent, by Proposition~\ref{Proposition binary sums} (ii) the induced morphism
\[
f\DefEq\mus{F(\iota_A)}{F(\iota_B)}\colon F(A)+F(B) \to F(A+B)
\]
is a strong epimorphism.
It follows by the universal property of the coproduct that the diagram
\[
\xymatrix@C=7em{
& A+B \ar@{=}[ld] \ar@{=}[rd] \\
GF(A)+GF(B) \ar[r]_-{g \DefEq\mus{G(\iota_{F(A)})}{G(\iota_{F(B)})}} & G(F(A)+F(B)) \ar[r]_-{G(f)} & GF(A+B)
}
\]
commutes, and so since $G$ is coherent, by Proposition~\ref{Proposition binary sums} (ii), that the morphism~$g$ is an isomorphism.
This means that $G(f)$ is an isomorphism and hence---since~$G$ reflects isomorphisms---that $f$ is an isomorphism as required.
\end{proof}

\begin{theorem}\label{(FWACC)}
Let $\C$ be a regular Mal'tsev category with pushouts of split monomorphisms.
\begin{enumerate}
\item If $\C$ is algebraically coherent, then the change-of-base functor along any split epimorphism preserves finite colimits.
\item When $\C$ is, in addition, a cocomplete well-powered category in which filtered colimits commute with finite limits---for instance, $\C$ could be a variety---then if $\C$ is algebraically coherent, it is fibre-wise algebraically cartesian closed \FWACC.
\end{enumerate}
\end{theorem}
\begin{proof}
Let us start with the first statement. For each $X$ in \C, $\Pt_X(\C)$ is a Mal'tsev category by Example 2.2.15 in~\cite{Borceux-Bourn}, and it is algebraically coherent by Corollary~\ref{Corollary Fibres}. Then, by Theorem~\ref{theorem: Mal'tsev implies protomodular}, $\Pt_X(\C)$ is protomodular for each $X$.

Suppose now that $p$ is a split epimorphism in \C, with a section $s$. Then by Lemma~\ref{lemma: split monos} the functor $s^*$ reflects isomorphisms. Hence we can apply Lemma~\ref{Lemma Coproducts}, with $F=p^*$ and $G=s^*$, to prove that $p^*$ preserves binary coproducts. Finally, by Theorem~4.3 in~\cite{Gray2012}, $p^*$ preserves finite colimits.

Statement (b) follows from (a) again via Theorem~4.3 in~\cite{Gray2012}.
\end{proof}

\subsection{Action accessible categories.}
It is unclear to us how the notion of \emph{action accessibility} introduced in~\cite{BJ07} is related to algebraic coherence. The two conditions share many examples and counterexamples, but we could not find any examples that separate them. On the other hand, we also failed to prove that one implies the other, so for now the relationship between the two conditions remains an open problem.

%%%%%%%%%%%%%%%%%%%%%%%%%%%%%%%%%%%%%%%%%%%%%%%%%%%%%%%%%%%%
\section{Decomposition of the ternary commutator}\label{Ternary commutator}

It is known~\cite{HVdL} that for normal subgroups $K$, $L$ and $M$ of a group $X$, 
\[
[K,L,M]=[[K,L],M]\join [[L,M],K]]\join [[M,K],L]
\]
where the commutator on the left is defined as in Section~\ref{Subsection Diamond}. Since, by the so-called \emph{Three Subgroups Lemma}, any of the latter commutators is contained in the join of the other two, we see that 
\[
[K,L,M]=[[K,L],M]\join[[M,K],L].
\]
We shall prove that this result is valid in any algebraically coherent semi-abelian category.
This gives us a categorical version of the Three Subgroups Lemma, valid for \emph{normal} subobjects of a given object. Recall, however, that the usual Three Subgroups Lemma for groups works for arbitrary subobjects.

\begin{theorem}[Three Subobjects Lemma for normal subobjects]\label{3SO}
If $K$, $L$ and~$M$ are normal subobjects of an object $X$ in an algebraically coherent semi-abelian category, then 
\[
[K,L,M]=[[K,L],M]\join[[M,K],L].
\]
In particular, $[[L,M],K]\leq [[K,L],M]\join [[M,K],L]$.
\end{theorem}
\begin{proof}
First note that in the diagram
\[
\resizebox{\textwidth}{!}{
\xymatrix@C=4em{0 \ar[r] & (K\flat L)\cosmash (K\flat M) \pullback \ar@{.{ >>}}[d]_-{\alpha} \ar@{{ |>}->}[r]^{\iota_{K\flat L,K\flat M}} & (K\flat L)+ (K\flat M) \ar@{.{ >>}}[d]^-{\mus{K\flat \iota_{L}}{K\flat \iota_{M}}} \ar[r]^-{\sigma_{K\flat L,K\flat M}} & (K\flat L)\times (K\flat M) \ar@{=}[d] \ar[r] & 0\\
0 \ar[r] & A \ar@{{ |>}->}[r] & K\flat (L+ M) \ar[r]_-{\left\links\begin{smallmatrix} K\flat\mus{1_{L}}{0}\\ K\flat\mus{0}{1_{M}}\end{smallmatrix}\right\rechts} & (K\flat L)\times (K\flat M) \ar[r] & 0}}
\]
the middle arrow, and hence also the induced left hand side arrow $\alpha$, are strong epimorphisms by algebraic coherence in the form of Theorem~\ref{Theorem Bemol}.
Hence also in the diagram
\[
\resizebox{\textwidth}{!}{
\xymatrix@C=6em{0 \ar[r] & B \pullback \ar@{.{ >>}}[d]_-{\beta} \ar@{{ |>}->}[r] & (K\flat L)\cosmash (K\flat M) \ar@{.{ >>}}[d]_-{\alpha} \ar@<.5ex>[r]^-{\mus{0}{1_{L}}\kappa_{K,L}\cosmash\mus{0}{1_{M}}\kappa_{K,M}} & L\cosmash M \ar@<.5ex>[l]^-{\eta_{L}\cosmash \eta_{M}} \ar@{=}[d] \ar[r] & 0\\
0 \ar[r] & K\cosmash L\cosmash M \ar@{{ |>}->}[r] & A \ar@<.5ex>[r] & L\cosmash M \ar@<.5ex>[l] \ar[r] & 0}}
\]
of which the bottom split short exact sequence is obtained via the ${3\times 3}$~diagram in Figure~\ref{3x3}, 
\begin{figure}[t]
\resizebox{\textwidth}{!}
{$\vcenter{\xymatrix@C=7em{
& 0 \ar[d] & 0 \ar[d] & 0 \ar[d] \\
0 \ar@{.>}[r] & K\cosmash L\cosmash M \ar@{{ |>}->}[d] \ar@{{ |>}.>}[r] & A \ar@{{ |>}->}[d] \ar@{.>}@<.5ex>[r] &  \ar@{.>}@<.5ex>[l] \ar@{.>}[r] L\cosmash M \ar@{{ |>}->}[d]^-{\iota_{L,M}} & 0\\
0 \ar[r] & K\cosmash(L+M) \ar@{-{ >>}}[d]_-{\left\links\begin{smallmatrix} K\cosmash\mus{1_{L}}{0}\\ K\cosmash\mus{0}{1_{M}}\end{smallmatrix}\right\rechts} \ar@{{ |>}->}[r]^-{j_{K,L+M}} & K\flat(L+M) \ar@<.5ex>[r]^-{\mus{0}{1_{L+M}}\kappa_{K,L+M}} \ar@{-{ >>}}[d]_-{\left\links\begin{smallmatrix} K\flat\mus{1_{L}}{0}\\ K\flat\mus{0}{1_{M}}\end{smallmatrix}\right\rechts} & L+M \ar@<.5ex>[l]^-{\eta_{L+M}} \ar@{-{ >>}}[d]^-{\sigma_{L,M}} \ar[r] & 0\\
0 \ar[r] & (K\cosmash L)\times (K\cosmash M) \ar@{{ |>}->}[r]_-{j_{K,L}\times j_{K,M}} \ar[d] & (K\flat L)\times (K\flat M) \ar[d] \ar@<.5ex>[r]^-{\mus{0}{1_{L}}\kappa_{K,L}\times\mus{0}{1_{M}}\kappa_{K,M}} & \ar@<.5ex>[l]^-{\eta_{L}\times \eta_{M}} L\times M \ar[d] \ar[r] & 0\\
& 0 & 0 & 0}}
$}
\caption{An alternative computation of $K\cosmash L\cosmash M$}\label{3x3}
\end{figure}
we have a vertical strong epimorphism on the left. Indeed, the left hand side vertical sequence in Figure~\ref{3x3} is exact by~\cite[Remark 2.8]{HVdL}, and the right hand one by definition of $L\cosmash M$. Its middle and bottom horizontal sequences are exact by Proposition~2.7 in~\cite{Actions} and because products preserve short exact sequences.

We have that $K\flat M = {(K\cosmash M) \join M}$ in $K+M$ and $K\flat L = (K\cosmash L) \join L$ in $K+L$ by~\cite[Proposition~2.7]{Actions}, so we may use Proposition~2.22 in~\cite{HVdL} to see that $(K\flat L)\cosmash (K\flat M)$ is covered by
\[
L\cosmash (K\flat M)+ (K\cosmash L)\cosmash (K\flat M)+ L\cosmash (K\cosmash L)\cosmash (K\flat M),
\]
which by further decomposition using~\cite[Proposition~2.22]{HVdL} gives us a strong epimorphism from $(L\cosmash M)+S$ to $(K\flat L)\cosmash (K\flat M)$, where $S$ is
\begin{align*}
& L\cosmash (K\cosmash M)+ L\cosmash M\cosmash (K\cosmash M)
+ (K\cosmash L)\cosmash M\\
&+ (K\cosmash L)\cosmash (K\cosmash M)+ (K\cosmash L)\cosmash M\cosmash (K\cosmash M)
+ L\cosmash (K\cosmash L)\cosmash M\\
& + L\cosmash (K\cosmash L)\cosmash (K\cosmash M)+ L\cosmash (K\cosmash L)\cosmash M\cosmash (K\cosmash M).
\end{align*}
Via the diagram
\[
\resizebox{\textwidth}{!}{
\xymatrix@C=7em{0 \ar[r] & (L\cosmash M)\flat S \pullback \ar@{.{ >>}}[d]_-{\gamma} \ar@{{ |>}->}[r]^-{\kappa_{L\cosmash M,S}} & (L\cosmash M)+ S \ar@{.{ >>}}[d] \ar@<.5ex>[r]^-{\mus{1_{L\cosmash M}}{0}} & L\cosmash M \ar@<.5ex>[l]^-{\iota_{L\cosmash M}} \ar@{=}[d] \ar[r] & 0\\
0 \ar[r] & B \ar@{{ |>}->}[r] & (K\flat L)\cosmash (K\flat M) \ar@<.5ex>[r]^-{\mus{0}{1_{L}}\kappa_{K,L}\cosmash\mus{0}{1_{M}}\kappa_{K,M}} & L\cosmash M \ar@<.5ex>[l]^-{\eta_{L}\cosmash \eta_{M}} \ar[r] & 0}}
\]
it induces a strong epimorphism $\beta\gamma$ from $(L\cosmash M)\flat S = S\vee (L\cosmash M)\cosmash S$ to $K\cosmash L\cosmash M$. We thus obtain a strong epimorphism from $S + (L \cosmash M) \cosmash S$ to $K\cosmash L\cosmash M$. Considering $K$, $L$ and $M$ as subobjects of~$X$ now, we take the images of the induced arrows to~$X$ to see that $[K,L,M]=\overline{S}\join [[L,M],\overline{S}]$ in~$X$, where $\overline{S}$ is the image of ${S\to X}$. Now $\overline{S}$, being
\begin{align*}
&[L,[K,M]]\join [L,M,[K,M]]\join [[K,L],M]\\
&\join [[K,L],[K,M]]\join[[K,L],M,[K,M]]\join [L,[K,L],M]\\
&\join [L,[K,L],[K,M]]\join [L,[K,L],M,[K,M]],
\end{align*}
is contained in $[L,[K,M]]\join [[K,L],M]$ by Proposition~2.21 in~\cite{HVdL}, using \SH\ in the form of~\cite[Theorem~4.6]{HVdL}, using \NH\ and the fact that $K$, $L$ and~$M$ are normal. Indeed, $[L,M,[K,M]]\leq [L,X,[K,M]]\leq [L, [K,M]]$ by \SH; note that Theorem~4.6 in~\cite{HVdL} is applicable because $[K,M]$ is normal in $X$ by \NH\ and normality of $K$ and $M$. Similarly, $[[K,L],[K,M]]\leq [L,[K,M]]$, and likewise for the other terms of the join. 

Hence 
\begin{align*}
[K,L,M]&=\overline{S}\join [[L,M],\overline{S}]\\
&\leq [L,[K,M]]\join [[K,L],M]\join \bigl[[L,M],[L,[K,M]]\join [[K,L],M]\bigr]\\
&\leq [L,[K,M]]\join [[K,L],M]
\end{align*}
because $[L,[K,M]]\join [[K,L],M]$ is normal as a join of normal subobjects, so that
\[
\bigl[[L,M],[L,[K,M]]\join [[K,L],M]\bigr]\leq [L,[K,M]]\join [[K,L],M]
\]
by~\cite[Proposition~6.1]{MM-NC}. Since the other inclusion
\[
[K,L,M]\geq [L,[K,M]]\join [[K,L],M]
\]
holds by~\cite[Proposition~2.21]{HVdL}, this finishes the proof.
\end{proof}

As a consequence, in any algebraically coherent semi-abelian category, the two natural, but generally non-equivalent, definitions of \emph{two-nilpotent object}---$X$ such that either $[X,X,X]$ or $[[X,X],X]$ vanishes, see also Section~\ref{Subsection Diamond}---coincide:

\begin{corollary}
In an algebraically coherent semi-abelian category,
\[
[X,X,X]=[[X,X],X]
\]
holds for all objects $X$.\noproof
\end{corollary}

Note that, since one of its entries is $X$, the commutator on the right is a normal subobject of $X$, which makes it coincide with the Huq commutator $[[X,X]_{X},X]_{X}$. Furthermore, by Proposition~2.2 in~\cite{Gran-VdL}, this commutator vanishes is and only if the Smith commutator $[[\nabla_{X},\nabla_{X}],\nabla_{X}]$ does. This implies that the normal subobject $[X,X,X]$ is the normalisation of the equivalence relation $[[\nabla_{X},\nabla_{X}],\nabla_{X}]$.

%%%%%%%%%%%%%%%%%% Summary
%\pagebreak
\section{Summary of results in the semi-abelian context}
In this section we give several short summaries.
We begin with a summary of conditions that follow from algebraic coherence for a semi-abelian category $\C$:
\begin{enumerate}
\item preservation of Higgins and Huq commutators, normal closures and cokernels by change-of-base functors with respect to the fibration of points, see~\ref{thm summary};
\item \SH\ and \NH, see~\ref{(SH) + (NH)};
\item as a consequence---see~\ref{Peri-abelian}---the category $\C$ is necessarily peri-abelian and thus satisfies the universal central extension condition;
\item \SSH, see~\ref{(SSH)};
\item strong protomodularity, see~\ref{(SSH) => (SP)};
\item fibre-wise algebraic cartesian closedness \FWACC, if $\C$ is a variety, see~\ref{(FWACC)} and~\cite{Bourn-Gray,Gray2012};
\item $[K,L,M]=[[K,L],M]\join[[M,K],L]$ for $K$, $L$, $M\normal X$, see~\ref{3SO}.
\end{enumerate}

Next we give a summary of semi-abelian categories which are algebraically coherent.
These include all abelian categories; all \emph{categories of interest} in the sense of Orzech: (all subvarieties of) groups, the varieties of Lie algebras, Leibniz algebras, rings, associative algebras, Poisson algebras; cocommutative Hopf algebras over a field of characteristic zero; $n$-nilpotent or $n$-solvable groups, rings, Lie algebras; internal reflexive graphs, categories and (pre)crossed modules in such; compact Hausdorff algebras over such; arrows, extensions and central extensions in such---note, however, that the latter two categories are only homological in general.

Finally we give a summary of semi-abelian categories which are not algebraically coherent.
These include (commutative) loops, digroups, non-associative rings, Jordan algebras.

\section*{Acknowledgement}

We thank the referee for careful reading of our manuscript and detailed comments on it, which lead to the present improved version.

%\bibliography{tim}
%\bibliographystyle{amsplain}
%%%% .bbl:

\providecommand{\noopsort}[1]{}
\providecommand{\bysame}{\leavevmode\hbox to3em{\hrulefill}\thinspace}
\providecommand{\MR}{\relax\ifhmode\unskip\space\fi MR }
\providecommand{\MRhref}[2]{%
  \href{http://www.ams.org/mathscinet-getitem?mr=#1}{#2}
}
\providecommand{\href}[2]{#2}

\end{document}